\documentclass[a4paper,11pt]{amsart}

\usepackage{amssymb,amsmath,amsthm}
\usepackage[all]{xy}
\usepackage{hhline}
\usepackage{braket}

\allowdisplaybreaks

{\theoremstyle{definition}\newtheorem{df}{Definition}[section]}
{\theoremstyle{definition}\newtheorem{rem}[df]{Remark}}
{\theoremstyle{definition}}
{\theoremstyle{definition}\newtheorem{const}[df]{Construction}}
{\theoremstyle{definition}\newtheorem{notation}[df]{Notation}}
\newtheorem{prop}[df]{Proposition}
\newtheorem{lem}[df]{Lemma}
\newtheorem{thm}[df]{Theorem}
\newtheorem{prob}[df]{Problem}
\newtheorem{cor}[df]{Corollary}

\newcommand{\bbZ}{\mathbb{Z}}

\newcommand{\bbR}{\mathbb{R}}
\newcommand{\bbC}{\mathbb{C}}

\newcommand{\Aut}{\mathrm{Aut}}
\newcommand{\id}{\mathrm{id}}

\newcommand{\minor}[1]{\vert#1\vert}
\newcommand{\calM}{\mathcal{M}}

\newcommand{\calI}{\mathcal{I}}
\newcommand{\calJ}{\mathcal{J}}
\newcommand{\calZ}{\mathcal{Z}}
\newcommand{\calC}{\mathcal{C}}
\newcommand{\calB}{\mathcal{B}}
\newcommand{\calF}{\mathcal{F}}

\newcommand{\calA}{\mathcal{A}}

\newcommand{\co}{\colon\thinspace}

\newcommand{\relint}{\mathrm{relint}}

\newcommand{\ccp}[1]{\bbC P^{#1}\sharp\bbC P^{#1}}
\newcommand{\cp}[1]{\bbC P^{#1}}

\newcommand{\depth}{\mathrm{depth}}

\newcommand{\sgn}{\mathrm{sgn}}

\newcommand{\pr}{\mathrm{pr}}

\newcommand{\calK}{\mathcal{K}}

\newcommand{\frD}{\mathfrak{D}}
\newcommand{\frF}{\mathfrak{F}}

\title{On the cohomology equivalences between bundle-type quasitoric manifolds over a cube}
\author{Sho Hasui}
\date{}
\address{Department of Mathematics, Faculty of Science, Kyoto University, Sakyo-ku, Kyoto 606-8502, Japan}
\email{s.hasui@math.kyoto-u.ac.jp}

\begin{document}

\maketitle

\begin{abstract}
The aim of this article is to establish the notion of bundle-type quasitoric manifolds and provide two classification results on them:
(1) ($\ccp{2}$)-bundle type quasitoric manifolds are weakly equivariantly homeomorphic if their cohomology rings are isomorphic, and (2) quasitoric manifolds over $I^3$ are homeomorphic if their cohomology rings are isomorphic.
In the latter case, there are only four quasitoric manifolds up to weakly equivariant homeomorphism which are not bundle-type.
\end{abstract}

\section{Introduction}

A quasitoric manifold $M$ over a simple polytope $P$, which was introduced by Davis and Januszkiewicz \cite{DJ91},
is a $2n$-dimensional smooth manifold with a locally standard $T^n=(S^1)^n$-action for which the orbit space is identified with $P$.
%
%
%
Quasitoric manifolds are defined as a topological counterpart of toric varieties.
Actually, as the toric varieties are in one-to-one correspondence with the fans, the quasitoric manifolds over $P$ are in one-to-one correspondence with a kind of combinatorial objects, called characteristic maps on $P$.
Moreover, any smooth projective toric variety turns out to be a quasitoric manifold if we regard that $T^n$ acts on it through the inclusion to $(\bbC^{\times})^n$.

On the classification of quasitoric manifolds, Masuda posed the following \textit{cohomological rigidity problem for quasitoric manifolds} in \cite{M08} where he affirmatively solved the equivariant version of it.

\begin{prob}
Are two quasitoric manifolds homeomorphic if their cohomology rings are isomorphic as graded rings?
\end{prob}

Since then toric topologists have studied the topological classification of quasitoric manifolds from the viewpoint of cohomological rigidity,
and now we have some classification results which give partial affirmative answers for this problem.
First, the cohomological rigidity of quasitoric manifolds over the simplex $\Delta^n\,(n=1,2,\ldots)$ is shown in \cite{DJ91}.
Second, the cohomological rigidity of quasitoric manifolds over the convex $m$-gon ($m=4,5,\ldots$) is an immediate corollary of the classification theorem of Orlik and Raymond \cite{OR70}.
Third, over the product of two simplices, the cohomological rigidity is proved by Choi, Park, and Suh \cite{CPS12}.
Finally, over the dual cyclic polytope $C^n(m)^*$ ($n\geq 4$ or $m-n=3$), it is shown by the author \cite{H15}.
In addition, there are some results on the cohomological rigidity of Bott manifolds, a special subclass of quasitoric manifolds over cubes, by Choi, Masuda, and Suh \cite{CMS}, Choi \cite{Choi} and Choi, Masuda, and Murai \cite{CMM}.
On the other hand, we have found no counterexample to this problem.

In this article we mainly consider the cohomological rigidity of ``bundle-type'' quasitoric manifolds over the cube $I^n$, which we give the precise definition later.
Bundle-type quasitoric manifolds form a large subclass of quasitoric manifolds.
For instance, up to weakly equivariant homeomorphism, the equivariant connected sum $\ccp{2}$ is the only one quasitoric manifold over $I^2$ which is not bundle-type (Proposition \ref{only ccp{2}} and Remark \ref{rem:kappa^2}), and there are only four quasitoric manifolds over $I^3$ which are not bundle-type (Lemma \ref{lem:summary up to w.e.homeo.}).
Note that there are infinitely many quasitoric manifolds over $I^n$ ($n\geq 2$) up to weakly equivariant homeomorphism.

The goal of this article is to show the following two theorems.
Here a ($\mathbb{C}P^2\sharp\mathbb{C}P^2$)-bundle type quasitoric manifold means an iterated ($\mathbb{C}P^2\sharp\mathbb{C}P^2$)-bundle over a point equipped with a good torus action, of which the precise definition is given in Section 2.2.

\begin{thm}\label{main thm 1}
Suppose that there is a graded ring isomorphism $\varphi\co H^*(M';\bbZ)\to H^*(M;\bbZ)$ between the cohomology rings of two $(\ccp{2})$-bundle type quasitoric manifolds $M$ and $M'$.
Then there exists a weakly equivariant homeomorphism $f\co M\to M'$ which induces $\varphi$ in cohomology.
\end{thm}

\begin{thm}\label{main thm 2}
Suppose that there is a graded ring isomorphism $\varphi\co H^*(M';\bbZ)\to H^*(M;\bbZ)$ between the cohomology rings of two quasitoric manifolds $M$ and $M'$ over $I^3$.
Then there exists a homeomorphism $f\co M\to M'$ which induces $\varphi$ in cohomology.
\end{thm}

This article is organized as follows.
In Section 2, we review the basics of quasitoric manifolds and give the precise definitions of the terms bundle-type quasitoric manifold and so on.
In Section 3, we prove the key lemma of this article (Lemma \ref{lem:sublemma, filtration lemma 2-2}) and prove Theorem \ref{main thm 1}.
In Section 4, we classify the quasitoric manifolds over $I^3$ up to weakly equivariant homeomorphism.
Finally, we give the proof of Theorem \ref{main thm 2} in Section 5.

\section{Preliminaries}

\subsection{Basics of quasitoric manifolds}

First, let us begin with the definition of a quasitoric manifold.
The reader can find more detailed explanation in e.g. Buchstaber and Panov \cite{BP02} and \cite{H15}.
Here we always assume that $\bbC^n$ is equipped with the standard $T^n$-action, i.e. the action defined by $t z:=(t_1 z_1,\ldots,t_n z_n)$ where $t=(t_1,\ldots,t_n)\in T^n$ and $z=(z_1,\ldots,z_n)\in \bbC^n$ respectively.

For two $T^n$-spaces $X$ and $Y$, a map $f\co X\to Y$ is called \textit{weakly equivariant} if there exists $\psi\in\Aut(T^n)$ such that $f(t x)=\psi(t) f(x)$ for any $t\in T^n$ and $x\in X$, where $\Aut(T^n)$ denotes the group of continuous automorphisms of $T^n$.
We say a smooth $T^n$-action on a $2n$-dimensional differentiable manifold $M$ is \textit{locally standard}
if for each $z\in M$ there exists a triad $(U,V,\varphi)$ consisting of a $T^n$-invariant open neighborhood $U$ of $z$,
a $T^n$-invariant open subset $V$ of $\bbC^n$, and a weakly equivariant diffeomorphism $\varphi\co U \to V$.

The orbit space of a locally standard $T^n$-action is naturally regarded as a \textit{manifold with corners}, by which we mean a Hausdorff space locally homeomorphic to an open subset of $(\bbR_{\geq 0})^n$ with the transition functions preserving the depth.
Here $\depth\,x$ of $x\in (\bbR_{\geq 0})^n$ is defined as the number of zero components of $x$.
By definition, for a manifold with corners $X$, we can define the \textit{depth of} $x\in X$ by $\depth\,x:=\depth\,\varphi(x)$ where $\varphi$ is an arbitrary local chart around $x$.
Then a map $f$ between two manifolds with corners is said to \textit{preserve the corners} if $\depth\circ f=\depth$.

An $n$-dimensional convex polytope is called \textit{simple} if it has exactly $n$ facets at each vertex.
We regard a simple polytope as a manifold with corners in the natural way, and define a quasitoric manifold as follows.

\begin{df}\label{df:qt mfd}
A \textit{quasitoric manifold over a simple polytope $P$} is a pair $(M,\pi)$ consisting of a $2n$-dimensional smooth manifold $M$ equipped with a locally standard $T^n$-action
and a continuous surjection $\pi\co M\to P$ which descends to a homeomorphism from $M/T^n$ to $P$ preserving the corners.
We omit the projection $\pi$ unless it is misleading.
\end{df}


Next we recall the two ways to construct a quasitoric manifold.
In this section $P$ always denotes an $n$-dimensional simple polytope with exactly $m$ facets and $\calF(P)$ denotes the face poset of $P$.
In addition, we define $\mathfrak{T}^n$ as the set of subtori of $T^n$.

\begin{df}\label{df:ch map}
A \textit{characteristic map on $P$} is a map $\ell\co \calF(P)\to \mathfrak{T}^n$ such that
\begin{itemize}
\item[(i)] $\dim \ell(F)=n-\dim F$ for each face $F$,
\item[(ii)] $\ell(F)\subseteq \ell(F')$ if $F'\subseteq F$, and
\item[(iii)] if a face $F$ is the intersection of distinct $k$ facets $F_1,\ldots,F_k$, then the inclusions $\ell(F_i)\to\ell(F)$ ($i=1,\ldots,k$) induce an isomorphism $\ell(F_1)\times\cdots\times\ell(F_k)\to\ell(F)$.
\end{itemize}
\end{df}

\begin{rem}
For each face $F$ of $P$, we denote the relative interior of $F$ by $\relint\,F$.
Given a quasitoric manifold $M$ over $P$, then we obtain a characteristic map $\ell_M$ on $P$ by
\[ \ell_M(F):= (T^n)_z \]
where $z$ is an arbitrary point of $\pi^{-1}(\relint\, F)$ and $(T^n)_z$ denotes the isotropy subgroup at $z$.
Actually we can easily check the conditions of Definition \ref{df:ch map} by the locally-standardness.
\end{rem}

\begin{const}\label{const:ch map}
For each point $q\in P$, we denote the minimal face containing $q$ by $G(q)$.
Then we obtain a quasitoric manifold $(M(\ell),\pi)$ over $P$ by setting
\[M(\ell):=(T^n\times P)/\!\sim_{\ell},\]
where $(t_1,q_1)\sim_{\ell}(t_2,q_2)$ if and only if $q_1=q_2$ and $t_1t_2^{-1}\in \ell(G(q_1))$, and $\pi\co M(\ell)\to P$ denotes the map induced by $\pr_2\co T^n\times P\to P$.
Obviously the $T^n$-action on $T^n\times P$ by multiplication on the first component descends to a $T^n$-action on $M(\ell)$.

We can define a differentiable structure on $M(\ell)$ as follows.
We regard $P$ as a subset of $\bbR^n$ and denote the hyperplane $\{(x_1,\ldots,x_n)\in\bbR^n\,\vert\,x_i=0\}$ by $H_i$ ($i=1,\ldots,n$).
For a vertex $v$ of $P$, we denote by $U_v$ the open subset of $P$ obtained by deleting all faces not containing $v$ from $P$, and take $n$ facets $F_1,\ldots,F_n$ of $P$ such that $v=F_1\cap\cdots\cap F_n$.
Additionally, we take an affine transformation $\bar{\varphi}_v$ of $\bbR^n$ which maps $U_v$ onto an open subset of $(\bbR_{\geq 0})^n$ and $F_i$ into $H_i$.
If we take an automorphism $\psi_v$ of $T^n$ which maps $\ell(F_{i})$ into the $i$-th coordinate subtorus for each $i=1,\ldots,n$, then the map $\psi_v\times\bar{\varphi}_v\co T^n\times U_v\to T^n\times (\bbR_{\geq 0})^n$ descends to a weakly equivariant homeomorphism $\varphi_v$ from $\pi^{-1}(U_v)$ to some $T^n$-invariant open subset of $\bbC^n$.
We can check that the atlas $\{(\pi^{-1}(U_v),\varphi_v)\}$ gives a differentiable structure on $M(\ell)$.
Clearly the $T^n$-action on $M(\ell)$ is locally standard and the orbit space is identified with $P$, i.e. $M(\ell)$ is a quasitoric manifold over $P$.
Moreover, by definition, we have $\ell=\ell_{M(\ell)}$.
\end{const}

In this article, we define an isomorphism of quasitoric manifolds as follows:
for two quasitoric manifolds $(M,\pi)$ and $(M',\pi')$ over $P$, a map $f\co M\to M'$ is called an \textit{isomorphism of quasitoric manifolds} if it is a $T^n$-equivariant homeomorphism such that $\pi'\circ f=\pi$.

By using the blow-up method of Davis \cite{D78}, we see that for any quasitoric manifold $M$ over $P$ there exists a $T^n$-equivariant surjection $T^n\times P\to M$ which descends to an isomorphism $M(\ell_M)\to M$ of quasitoric manifolds.
Thus we obtain the following.

\begin{prop}\label{prop:fundamental}
The correspondence $\ell \mapsto M(\ell)$ gives a bijection from the set of characteristic maps on $P$ to the set of isomorphism classes of quasitoric manifolds over $P$, and the inverse is given by $M\mapsto \ell_M$.
\end{prop}

The second way to construct a quasitoric manifold uses a characteristic matrix and a moment-angle manifold.
Below the term \textit{facet labeling of} $P$ means a bijection from $\{1,\ldots,m\}$ to the set of the facets of $P$.

\begin{df}
An $(n\times m)$-matrix $\lambda=(\lambda_1,\ldots,\lambda_m)$ of integers is called a \textit{characteristic matrix on $P$ with respect to the facet labeling $F_1,\ldots,F_m$} if it satisfies the following \textit{nonsingularity condition}:
if $F_{i_1},\ldots,F_{i_n}$ meet at a vertex, then $\det(\lambda_{i_1},\ldots,\lambda_{i_n})=\pm 1$.
%
\end{df}

Hereafter, unless mentioned otherwise, we fix a facet labeling $F_1,\ldots,F_m$ of $P$.

\begin{rem}
Given a characteristic matrix $\lambda$ on $P$, we can define a characteristic map $\ell_{\lambda}$ by
\[ \ell_\lambda(F_{i_1}\cap \ldots \cap F_{i_k}):= \mathrm{im}\,(\lambda_{i_1},\ldots,\lambda_{i_k}) \]
where we identify $S^1$ with $\bbR/\bbZ$ and regard $(\lambda_{i_1},\ldots,\lambda_{i_k})$ as a homomorphism from $T^k$ to $T^n$.
Obviously, any characteristic map is obtained from some characteristic matrix in this way.
\end{rem}

\begin{const}\label{const:moment-angle mfd}
Let $K_P$ be the simplicial complex on $[m]:=\{1,\ldots,m\}$ defined by $K_P:=\{J_F\,|\,F\in\calF(P)\}$ where $J_F:=\{i\in[m]\,|\,F\subseteq F_i\}$.
We regard $D^2$ as the unit disc of $\bbC$ and define
$$ (D^2,S^1)^{J} := \{ (z_1,\ldots,z_m) \in (D^2)^m \,\vert\, \minor{z_i}=1\text{ if } i\not\in J \} $$
for each $J \subseteq [m]$.
Then the \textit{moment-angle manifold} $\calZ_P$ is defined as the union
$$\bigcup_{J \in K_P}(D^2,S^1)^{J} \subseteq (D^2)^m,$$
which is equipped with the $T^m$-action defined by $(t_1,\ldots,t_m)\cdot(z_1,\ldots,z_m)=(t_1 z_1,\ldots,t_m z_m)$.

We can define an embedding $\varepsilon\co P\to \calZ_P$ as follows.
Denote the barycentric subdivision of $K_P$ by $K'_P$.
If we take $b_F\in \relint\, F$ for each face $F$, then the correspondence $J_F\mapsto b_F$ gives a triangulation of $P$ by $K'_P$.
Then we define $\varepsilon\co |K'_P|\to \calZ_P$ so that $\varepsilon(J_F)=(c_1(F),\ldots,c_m(F))$ for the vertices and it restricts to an affine map on each simplex, where $c_i(F)=0$ if $F\subseteq F_i$ and $c_i(F)=1$ otherwise.
Note that $\varepsilon$ descends to a homeomorphism from $P\cong|K'_P|$ to $\calZ_P/T^m$.
If we define $G(q)$ as in Construction \ref{const:ch map} and $\ell_P\co \calF(P)\to\mathfrak{T}^m$ by
$$ \ell_P(F):=\{(t_1,\ldots,t_m)\in T^m\,\vert\, t_i =1\text{ if }F\not\subseteq F_i)\}, $$
then the correspondence $(t,q)\mapsto t\cdot \varepsilon(q)$ gives an equivariant homeomorphism from $ (T^m\times P)/\!\sim$ to $\calZ_P $, where $(t_1,q_1)\sim(t_2,q_2)$ if and only if $q_1=q_2$ and $t_1t_2^{-1}\in \ell_P(G(q_1))$.
Moreover, we can define a differentiable structure on $(T^m\times P)/\!\sim\ \cong \calZ_P$ in the same way as Construction \ref{const:ch map}, and then the $T^m$-action on $\calZ_P$ is smooth.

Let $\lambda$ be a characteristic matrix on $P$.
If we regard $\lambda$ as a homomorphism from $T^m$ to $T^n$, then we can check that $T_\lambda:=\ker \lambda$ acts on $\calZ_P$ freely.
Thus we obtain a manifold $M(\lambda):=\calZ_P/T_\lambda$ with a smooth action of $T^m/T_{\lambda}\cong T^n$ where the isomorphism is induced by $\lambda$.
We can easily check that this $T^n$-action on $M(\lambda)$ is locally standard.
Actually, the map $\lambda\times\id_P\co T^m\times P\to T^n\times P$ descends to an equivariant diffeomorphism from $M(\lambda)$ to $M(\ell_\lambda)$.
We define $\pi\co M(\lambda)\to P$ as the composite of the quotient map $M(\lambda)\to M(\lambda)/T^n=\calZ_P/T^m$ and $\varepsilon^{-1}$, and then $(M(\lambda),\pi)$ is a quasitoric manifold over $P$.
%
\end{const}

Clearly, we have the following proposition.

\begin{prop}
For a characteristic matrix $\lambda$ on $P$, the two quasitoric manifolds $M(\lambda)$ and $M(\ell_\lambda)$ are smoothly isomorphic.
\end{prop}

\begin{df}
For a quasitoric manifold $M$ over $P$, a \textit{characteristic matrix of $M$} means a characteristic matrix $\lambda$ on $P$ such that $M(\lambda)$ is isomorphic to $M$.
In other words, $\lambda$ is called a characteristic matrix of $M$ if $\ell_{\lambda}=\ell_M$.
\end{df}


Next we consider the weakly equivariant homeomorphisms between quasitoric manifolds.
We denote by $[m]_{\pm}$ the set of $2m$ integers $\pm1,\ldots,\pm m$ and regard that $\bbZ/2$ acts on $[m]_\pm$ by multiplication with $-1$.
Additionally, we define a map $\sgn\co [m]_\pm\to \bbZ/2$ so that $x=\sgn(x)\cdot|x|$ where we identify $\bbZ/2$ with the multiplicative group $\{\pm 1\}$.

\begin{df}\label{df:iota}
We define $R_m$ as the group of ($\bbZ/2$)-equivariant permutations of $[m]_\pm$ and $p\co R_m\to \mathfrak{S}_m$ as the canonical surjection to the symmetric group.
In addition, we define $\iota\co R_m\to GL_m(\bbZ)$ so that $e_i\cdot\iota(\rho)=\sgn_i(\rho)\cdot e_{\sigma(i)}$ ($i=1,\ldots,m$) where $e_1,\ldots,e_m$ denote the standard basis of $\bbZ^m$,  $\sigma:=p(\rho)$, and $\sgn_i(\rho):=\sgn(\rho(i))$.
\end{df}

\begin{rem}
The map $\iota\co R_m\to GL_m(\bbZ)$ defined above is an antihomomorphism.
Actually, if we take $\rho_i\in R_m$ and put $\sigma_i:=p(\rho_i)$ ($i=1,2$), we can check $\iota(\rho_1\circ \rho_2)=\iota(\rho_2)\cdot\iota(\rho_1)$ as follows.
For a fixed $j\in\{1,\ldots,m\}$, if we put $k:=\sigma_2(j)$, then  $\sgn_j(\rho_1\circ\rho_2)=\sgn_j(\rho_2)\cdot\sgn_k(\rho_1)$.
Therefore
\begin{align*}
e_j\cdot\iota(\rho_1\circ\rho_2)&=\sgn_j(\rho_1\circ\rho_2)\cdot e_{\sigma_1\circ\sigma_2(j)}=\sgn_j(\rho_2)\cdot(\sgn_k(\rho_1)\cdot e_{\sigma_1(k)}) \\ &=\sgn_j(\rho_2)\cdot (e_k\cdot\iota(\rho_1))=(\sgn_j(\rho_2)\cdot e_{\sigma_2(j)})\cdot\iota(\rho_1)=e_j\cdot\iota(\rho_2)\cdot\iota(\rho_1).
\end{align*}
\end{rem}

\begin{df}\label{df:action on Lambda_P}
For a simple polytope $P$, we denote by $\Aut(P)$ the group of combinatorial self-equivalences of $P$ and regard it as a subgroup of the symmetric group $\mathfrak{S}_m$ by using the facet labeling.
Then we denote by $R(P)$ the subgroup $p^{-1}(\Aut(P))$ of $R_m$.
Moreover, we define $\Lambda_P$ as the set of characteristic matrices on $P$ and a left action of $GL_n(\bbZ)\times R(P)$ on $\Lambda_P$ by $(\psi,\rho)\cdot \lambda:=\psi\cdot\lambda\cdot\iota(\rho)$.
\end{df}

\begin{df}\label{df:rep of wehomeo}
Let $P$ be a simple polytope, $\lambda$, $\lambda'$ be characteristic matrices on $P$, and $f\co M(\lambda)\to M(\lambda')$ be a weakly equivariant homeomorphism.
We denote by $\bar{f}$ the corner-preserving self-homeomorphism of $P$ induced by $f$.
Then a pair $(\psi,\rho)\in GL_n(\bbZ)\times R(P)$ is called the \textit{representation of} $f$ if the following (i), (ii), and (iii) hold.
\begin{itemize}
\item[(i)] $f(t x)=\psi(t)f(x)$ for any $t\in T^n$ and $x\in M(\lambda)$, where we identify $GL_n(\bbZ)$ with $\Aut(T^n)$ through the left action on $\bbR^n/\bbZ^n=T^n$.
\item[(ii)] If we denote by $\sigma\!_f$ the combinatorial self-equivalence of $P$ induced by $\bar{f}$, then $\sigma\!_f=p(\rho)$.
\item[(iii)] $\lambda'=(\psi,\rho)\cdot \lambda$.
\end{itemize}
\end{df}

It is easy to see that for any weakly equivariant homeomorphism $f\co M(\lambda)\to M(\lambda')$ there exists a unique representation of $f$.
Conversely, we have the following proposition.

\begin{prop}\label{realization of rep}
For any pair $(\psi,\rho)\in GL_n(\bbZ)\times R(P)$ and a characteristic matrix $\lambda$ on $P$, there exists a weakly equivariant homeomorphism $f\co M(\lambda)\to M(\lambda')$ of which the representation is $(\psi,\rho)$.
Here $\lambda'$ denotes the characteristic matrix $(\psi,\rho)\cdot\lambda$ on $P$.
\end{prop}

\begin{proof}
First, by using the triangulation of $P$ given in Construction \ref{const:moment-angle mfd}, we can construct a corner-preserving self-homeomorphism $\bar{f}$ of $P$ which induces $p(\rho)$.
Since $\lambda'=\psi\cdot \lambda\cdot \iota(\rho)$, we see $\psi(\ell(F))\subseteq \ell'(\sigma(F))$ for each face $F$ of $P$, where $\sigma:=p(\rho)$, $\ell:=\ell_\lambda$ and $\ell':=\ell_{\lambda'}$.
It implies that $(\psi(t_1),\bar{f}(q_1))\sim_{\ell'}(\psi(t_2),\bar{f}(q_2))$ if $(t_1,q_1)\sim_\ell(t_2,q_2)$, where $\sim_{\ell}$ and $\sim_{\ell'}$ are defined in the same way as Construction \ref{const:ch map}.
Thus we see that $\psi\times\bar{f}\co T^n\times P\to T^n\times P$ descends to a weakly equivariant homeomorphism $f\co M(\ell)\to M(\ell')$, of which the representation is obviously $(\psi,\rho)$.
%
\end{proof}

\begin{cor}\label{classification up to wehomeo}
If we denote by $\calM_P^\mathrm{weh}$ the set of weakly equivariant homeomorphism classes of quasitoric manifolds over $P$, then the correspondence $\lambda\mapsto M(\lambda)$ gives a bijection from $\Lambda_P/(GL_n(\bbZ)\times R(P))$ to $\calM_P^\mathrm{weh}$.
\end{cor}


Then we consider the cohomology ring of a quasitoric manifold $M=M(\lambda)$ over $P$.
The following  computation is due to \cite{DJ91}.

Let us define the \textit{Davis-Januszkiewicz space} $DJ_P$ as the union
$$\bigcup_{J \in K_P}BT^J \subseteq BT^m=(\bbC P^{\infty})^m$$
where $ BT^{J} := \{ (y_1,\ldots,y_m) \in BT^m \,\vert\, y=*\text{ if }i\not\in J \} $ and $*$ denotes the basepoint of $\bbC P^{\infty}$.
$K_P$ is the simplicial complex defined in Construction \ref{const:moment-angle mfd}.
Denote the Borel constructions of $M$ and $\calZ_P$ by $\calB_{T^n} (M)$ and $\calB_{T^m} (\calZ_P)$ respectively,
i.e. $\calB_{T^n} (M)$ (resp. $\calB_{T^m} (\calZ_P)$) denotes the quotient of $ET^n\times M$ (resp. $ET^m\times \calZ_P$) by the action of $T^n$ (resp. $T^m$) defined by $t\cdot(x,y):=(x t,t^{-1}y)$.
Then we have a homotopy commutative diagram
\[
\xymatrix {
\calZ_P \ar[1,0] \ar[0,1] & M \ar[1,0] \\
\calB_{T^m} (\calZ_P) \ar[1,0] \ar[0,1] & \calB_{T^n} (M) \ar[1,0] \\
BT^m \ar[0,1]^{B\lambda} & BT^n
}
\]
where the columns are fiber bundles, the middle horizontal map is a homotopy equivalence, and the bottom one is the map induced by $\lambda\co T^m\to T^n$.
By using homotopy colimit, we can construct a homotopy equivalence from $DJ_P$ to $\calB_{T^m} (\calZ_P)$ such that the diagram
\[
\xymatrix {
& \calB_{T^m} (\calZ_P)  \ar[1,0] \\
DJ_P \ar[0,1] \ar[-1,1] &  BT^m 
}
\]
commutes up to homotopy, where the horizontal arrow is the inclusion.

Thus we obtain the following theorem.

\begin{thm}[Davis and Januszkiewicz]\label{thm:qt mfd is a homotopy fiber}
Let $P$ be an $n$-dimensional simple polytope with $m$ facets and $\lambda$ be a characteristic matrix on $P$.
Then $M(\lambda)$ is the homotopy fiber of the map $B\lambda\circ \mathrm{incl}\co DJ_P \to BT^n$,
where $\mathrm{incl}$ denotes the inclusion into $BT^m$.
\end{thm}

Since it is also shown by Davis and Januszkiewicz (in the proof of \cite[Theorem 3.1]{DJ91}) that any quasitoric manifold has a CW structure without odd dimensional cells,
we immediately obtain the following corollary.

\begin{cor}[Davis and Januszkiewicz]\label{cor:cohomology ring of a quasitoric manifold}
Let $P$ be an $n$-dimensional simple polytope with $m$ facets, $\lambda=(\lambda_{i,j})$ be a characteristic matrix on $P$, and put $M:=M(\lambda)$.
Then the integral cohomology ring of $M$ is given by
\[ H^*(M;\bbZ)=\bbZ[v_1,\ldots,v_m] /(\calI_P+\calJ_{\lambda}).\]
Here $v_i:=j^* t_i \in H^2(M;\bbZ)\,(i=1,\ldots,m)$ where $j\co M\rightarrow DJ_P$ is the inclusion of fiber and $t_i$'s are the canonical basis of $H^2(DJ_P;\bbZ)$, and $\calI_P$, $\calJ_{\lambda}$ are the ideals below:
\begin{align*}
\mathcal{I}_P &= (v_{i_1}\cdots v_{i_k}\,\vert \, F_{i_1}\cap \ldots \cap F_{i_k}=\emptyset ),\\
\mathcal{J}_{\lambda} &= (\lambda_{i,1}v_1+\cdots +\lambda_{i,m}v_m\, \vert \, i=1,\ldots,n).
\end{align*}
\end{cor}

\begin{lem}\label{generators of cohomology}
For each $i=1,\ldots,m$, the generator $v_i\in H^2(M;\bbZ)$ of Corollary $\ref{cor:cohomology ring of a quasitoric manifold}$ is equal to the Poincar\'{e} dual of the submanifold $M_i:=\pi^{-1}(F_i)$ .
\end{lem}

We make some preparations before the proof of this lemma.
For the sake of simplicity, we make the following conventions.
\begin{itemize}
\item Unless otherwise mentioned, a space means a Hausdorff space and a map means a continuous map.
An action is also assumed to be continuous.
\item A structure group $\frF$ of a fiber bundle with fiber $F$ is always assumed to act on $F$ effectively.
Moreover, $\frF$ is assumed to have the following property: for a space $X$ and a possibly non-continuous map $f\co X \to \frF$, $f$ is continuous if the map $X\times F\to F$ defined by $(x,y)\mapsto f(x)\cdot y$ is continuous.
\end{itemize}

\begin{df}\label{df:equivariant bundle}
Let $G$, $\frF$ be two topological groups and regard that $\frF$ acts on a space $F$.
A \textit{$G$-equivariant fiber bundle with fiber $F$ and structure group $\frF$} is a map $p\co E\to B$ between $G$-spaces satisfying the following conditions:
\begin{itemize}
\item[(i)] $p$ is a fiber bundle with fiber $F$ and structure group $\frF$;
\item[(ii)] $p$ is $G$-equivariant;
\item[(iii)] for each $g\in G$ and $x\in B$, if we take local trivializations $\phi\co U\times F\to p^{-1}(U)$ and $\phi'\co U'\times F\to p^{-1}(U')$ around $x$ and $g x$ respectively, then there exists $f\in\frF$ such that the following diagram commutes.
\[\xymatrix{
p^{-1}(x) \ar[r]^g & p^{-1}(g x) \\
\{x\}\times F \ar[u]^{\phi|_{\{x\}\times F}} \ar[r]^{g\times f} & \{g x\}\times F \ar[u]_{\phi'|_{\{g x\}\times F}}
}\]
\end{itemize}

If $G$ is a Lie group, then a $G$-equivariant fiber bundle is called smooth if it is smooth as a fiber bundle, and the $G$-actions on the total space and the base space are smooth.
Here we say a fiber bundle is smooth if the fiber, the total space, and the base space are differentiable manifolds and the local trivializations can be chosen to be diffeomorphisms.
\end{df}


\begin{lem}\label{lem:equivariant bundle}
Let $G$ be a topological group, $H$ be a closed normal subgroup of $G$, $p\co E\to P$ be a $G$-equivariant fiber bundle with fiber $F$ and structure group $\frF$, and put $\overline{E}:=E/H$, $B:=P/H$.
If the quotient map $q\co P\to B$ is a principal $H$-bundle, then $\bar{p}\co\overline{E}\to B$ induced by $p$ can be equipped with a $G/H$-equivariant fiber bundle structure with fiber $F$ and structure group $\frF$ so that the quotient map $\tilde{q}\co E\to\overline{E}$ is a bundle map covering $q$.
\end{lem}

\begin{proof}
To summarize the setting of the lemma, we have the following commutative diagram.
\[\xymatrix{
F \ar[r] & E \ar[r]^{p} \ar[d]_{\tilde{q}} & P \ar[d]^{q} \\
& \overline{E} \ar[r]^{\bar{p}} & B
}\]
Let $\calA$ be the set consisting of triads $(U,s,\beta)$ where $U$ is an open subset of $B$, $s$ is a section of $q\co q^{-1}(U)\to U$, and $\beta\co V\times F\to p^{-1}(V)$ is a local trivialization of $p$ such that $s(U)\subseteq V$.
If we define $\phi_\alpha\co U\times F\to \bar{p}^{-1}(U)$ by $\phi_\alpha(x,y):=\tilde{q}\circ\beta(s(x),y)$ for each $\alpha=(U,s,\beta)\in\calA$, then it is clearly bijective.
Note that, since $\overline{E}$ is a quotient by a group action, $\tilde{q}$ is an open map and restricts to a quotient map $\tilde{q}^{-1}(W)\to W$ for any open subset $W$ of $\overline{E}$.
Since $\beta\circ(s\times\id_F)$ is a topological embedding and $\tilde{q}\circ\beta\circ(s\times\id_F)\circ\phi_\alpha^{-1}=\id_{\bar{p}^{-1}(U)}$ is continuous, $\phi_\alpha^{-1}$ is also continuous.
Thus we see that $\bar{p}$ is a fiber bundle with fiber $F$.

Next, let us show that the transition functions associated with the local trivializations $\{(U,\phi_\alpha)\}_{\alpha\in\calA}$ take values in $\frF$.
Take $\alpha=(U,s,\beta),\, \alpha'=(U',s',\beta')\in\calA$ and assume $U\cap U'\neq \emptyset$.
Due to the second convention made before Definition \ref{df:equivariant bundle}, we only have to show that for each $x\in U\cap U'$ there exists $f\in\frF$ such that $\phi_\alpha(x,y)=\phi_{\alpha'}(x,f\cdot y)$ for any $y\in F$.
Fix $x\in U\cap U'$ and take $h\in H$ such that $s'(x)=h\cdot s(x)$.
Since $p$ is a $G$-equivariant fiber bundle, there exists $f\in\frF$ such that $h\cdot\beta(s(x),y)=\beta'(s'(x),f\cdot y)$ for any $y\in F$.
Then, since $\tilde{q}(h\cdot\beta(s(x),y))=\tilde{q}(\beta(s(x),y))$, we have $\phi_\alpha(x,y)=\phi_{\alpha'}(x,f\cdot y)$ for any $y\in F$.
Thus we see that $\bar{p}$ is a fiber bundle with structure group $\frF$.

Finally, we show that the condition (iii) of Definition \ref{df:equivariant bundle} holds for $\bar{p}$.
Fix $g\in G$, $x\in B$ and take $\alpha=(U,s,\beta),\, \alpha'=(U',s',\beta')\in\calA$ so that $x\in U$, $gx\in U'$.
We can take $h\in H$ such that $s'(g x)=h\cdot (g\cdot s(x))$.
If we put $g':=h g$, since $p$ is a $G$-equivariant fiber bundle, there exists $f\in \frF$ such that $g'\cdot\beta(s(x),y)=\beta'(s'(g x),f\cdot y)$ for any $y\in F$.
Then, since $G$ acts on $\overline{E}$ via $G/H$, we have $g\cdot\phi_\alpha(x,y)=\phi_{\alpha'}(g x,f\cdot y)$.
Thus the proof is completed.
\end{proof}


\begin{proof}[proof of Lemma \ref{generators of cohomology}]
Fix $i\in\{1,\ldots,m\}$ and let $X$ be the inverse image of $M_i$ under the quotient map from $\calZ_P$ to $M=M(\lambda)$.
Then $X=\{(z_1,\ldots,z_m)\in\calZ_P\,\vert\,z_i=0\}$ (see Construction \ref{const:moment-angle mfd}).
We define a normal bundle $\nu(X)$ of $X$ in $\calZ_P$ by $\nu(X):=\{(z_1,\ldots,z_m)\in\calZ_P\,\vert\,|z_i|<1\}$ where the projection $\nu(X)\to X$ is given by $(z_1,\ldots,z_m)\mapsto (z_1,\ldots,z_{i-1},0,z_{i+1},\ldots,z_m)$.
Then $\pr_i\co \calZ_P\to D^2$ restricts to a bundle map $\nu(X) \to \mathrm{Int}\,D^2$ covering $X\to\{0\}$.
By Lemma \ref{lem:equivariant bundle}, since $\nu(X)$ is a $T^m$-equivariant vector bundle and $\calZ_P\to M(\lambda)$ is a principal $T_\lambda$-bundle, $\nu(M_i):=\nu(X)/T_\lambda$ gives a $T^n$-equivariant normal bundle of $M_i$ in $M$.
Moreover, by using Lemma \ref{lem:equivariant bundle} again, we see that $\calB_{T^n}(\nu(M_i))\to \calB_{T^n}(M_i)$ and $\calB_{T^m}(\nu(X))\to \calB_{T^m}(X)$ also have vector bundle structures.
Thus we have the following diagram where each square is a bundle map.
\[\xymatrix{
\nu(M_i) \ar[r] \ar[d] & \calB_{T^n}(\nu(M_i)) \ar[d] & \calB_{T^m}(\nu(X)) \ar[d] \ar[l] \ar[r] & \calB_{T^1}(\mathrm{Int}\,D^2) \ar[d] \\
M_i \ar[r] & \calB_{T^n}(M_i) & \calB_{T^m}(X) \ar[l] \ar[r] & B T^1
}\]

Then, let us put $(A,B)^c:=(A,A\setminus B)$, $\calB_{T^k}(A,B):=(\calB_{T^k}A,\calB_{T^k}B)$ for a pair $(A,B)$ of $T^k$-spaces, and consider the following commutative diagram.
\[\xymatrix{
H^2(\calB_{T^1}(D^2,\{0\})^c) \ar[d]_{r_1} \ar[r]^{\pr_i^*} & H^2(\calB_{T^m}(\calZ_P,X)^c) \ar[d] & H^2(\calB_{T^n}(M,M_i)^c) \ar[r] \ar[d] \ar[l]_{\cong} & 
H^2((M,M_i)^c) \ar[d]^{r_2} \\
H^2(\calB_{T^1} (D^2)) \ar[r]^{\pr_i} & H^2(\calB_{T^m}(\calZ_P)) & H^2(\calB_{T^n} (M)) \ar[r] \ar[l]_{\cong} & H^2(M) 
}\]
Here $H^*(\,\cdot\,)$ denotes the integral cohomology and each vertical arrow denotes the restriction.
Let us denote the Thom class of $\calB_{T^1}(\mathrm{Int}\,D^2)$ by $\tau$ and regard it as an element of $H^2(\calB_{T^1}(D^2,\{0\})^c)$ through the excision isomorphism.
Moreover, we denote the composite of the upper (resp. lower) horizontal arrows by $\gamma_1$ (resp. $\gamma_2$).
Due to the above diagram of bundle maps, $\gamma_1$ maps $\tau$ to the Thom class of $\nu(M_i)$, and therefore $r_2\circ\gamma_1(\tau)$ is the Poincar\'{e} dual of $M_i$.
Moreover, since $r_1(\tau)$ is the canonical generator of $H^2(\calB_{T^1} (D^2))\cong H^2(B T^1)$, $\gamma_2\circ r_1(\tau)=v_i$.
Thus the proof is completed.
\end{proof}

Note that, with the notation of Definition \ref{df:rep of wehomeo}, a weakly equivariant homeomorphism $f$ maps $\pi^{-1}(F_i)$ to $\pi'^{-1}(F_{\sigma\!_f(i)})$ for each $i=1,\ldots,m$, where $\pi$ (resp. $\pi'$) denotes the projection from $M(\lambda)$ (resp. $M(\lambda')$) to $P$.
By taking into account the orientations of the normal bundles, we have the following.

\begin{cor}\label{induced by rep}
Let $\lambda$, $\lambda'$ be two characteristic matrices on $P$ and $f\co M(\lambda)\to M(\lambda')$ be a weakly equivariant homeomorphism represented by $(\psi,\rho)\in GL_n(\bbZ)\times R(P)$.
Then, if we take generators $v_1,\ldots,v_m\in H^*(M(\lambda);\bbZ)$ and $v'_1,\ldots,v'_m\in H^*(M(\lambda');\bbZ)$ as in Corollary $\ref{cor:cohomology ring of a quasitoric manifold}$, we have
$$ f^*(v'_1, \ldots , v'_m)=(v_1, \ldots, v_m) \cdot \iota(\rho)^{-1}.$$
\end{cor}

To close this subsection, we introduce two theorems which we will use for the classification of quasitoric manifolds over $I^3$.

\begin{thm}[{\cite[Corollary 6.8]{DJ91}}]\label{thm:characteristic classes of a quasitoric manifold}
With the notation in Corollary $\ref{cor:cohomology ring of a quasitoric manifold}$, we have the following formulae for the total Stiefel-Whitney class and the total Pontrjagin class:
\begin{align*}
w(M) &= \prod_{i=1}^m (1+v_i),\\
p(M) &= \prod_{i=1}^m (1-v_i^2).
\end{align*}
\end{thm}

\begin{thm}[Jupp's classification of certain $6$-manifolds, \cite{Jup73}]\label{thm:classification theorem of 6-manifolds}
Let $M,N$ be closed, one-connected, smooth $6$-manifolds with torsion-free cohomology.
If a graded ring isomorphism $\alpha\co H^*(N;\bbZ)\to H^*(M;\bbZ)$ preserves the second Stiefel-Whitney classes and the first Pontrjagin classes,
then there exists a homeomorphism $f\co M\to N$ which induces $\alpha$ in cohomology.
\end{thm}

\subsection{Bundle-type quasitoric manifold}

Given a quasitoric manifold $M$, we denote by $\frD(M)$ the group of smooth automorphisms of $M$ equipped with the compact-open topology (recall that an isomorphism between quasitoric manifolds $(M,\pi)$ and $(M',\pi')$ means an equivariant homeomorphism $f\co M\to M'$ satisfying $\pi'\circ f=\pi$).
The following proposition is immediate from the definition of a smooth equivariant fiber bundle (Definition \ref{df:equivariant bundle}).

\begin{prop}\label{prop:associated torus action}
Let $M_i$ be a quasitoric manifold acted on by $T_i$ ($i=1,2$) and suppose that $p\co M \to M_2$ is a smooth $T_2$-equivariant fiber bundle with fiber $M_1$ and structure group $\frD(M_1)$.
Then there is a unique $T_1$-action on $M$ such that $t_1\cdot\phi(x,y)=\phi(x,t_1 y)$ for any local trivialization $\phi\co U\times M_1\to p^{-1}(U)$ of $p$ and $t_1\in T_1$.
Moreover, this action of $T_1$ on $M$ is smooth and commutes with the action of $T_2$.
\end{prop}

\begin{df}\label{df:qt bundle}
Let $M_i$ be a quasitoric manifold over $P_i$ acted on by $T_i$ ($i=1,2$).
Then a \textit{quasitoric $M_1$-bundle over $M_2$} is a smooth $T_2$-equivariant fiber bundle $p\co M \to M_2$ with fiber $M_1$, structure group $\frD(M_1)$, and total space equipped with the action of $T:=T_1\times T_2$ defined by $(t_1,t_2)\cdot x:=t_1(t_2 x)$, where the $T_1$-action is the one defined in Proposition \ref{prop:associated torus action}.
\end{df}

We prove later that the quasitoric bundle $M$ is a quasitoric manifold over $P_1\times P_2$.


\begin{df}
Let $\calM$ be a class of quasitoric manifolds and consider a sequence
\[\xymatrix{
B_l\ar[r]^(0.44){p_{l-1}}&B_{l-1}\ar[r]^(0.55){p_{l-2}}&\cdots\ar[r]^(0.52){p_1}&B_1\ar[r]^{p_0}&B_0
}\]
where $B_0$ is a point.
Then $B_l$ is called an \textit{$l$-stage $\calM$-bundle type quasitoric manifold} if $p_i$ is a quasitoric $M_i$-bundle for some $M_i\in \calM$ ($i=0,\ldots,l-1$).
\end{df}

Now the term $(\ccp{2})$-bundle type quasitoric manifold in Theorem \ref{main thm 1} is defined as follows:
let us define $\calM(\ccp{2})$ as the class of quasitoric manifolds which are homeomorphic to $\ccp{2}$, and use the term \textit{$(\ccp{2})$-bundle type quasitoric manifold} instead of $\calM(\ccp{2})$-bundle type quasitoric manifold.


%
%

Suppose that the facets of $P_i$ are labeled by $F_{i,1},\ldots,F_{i,m_i}$ and $\lambda_i$ is a characteristic matrix of $M_i$ with respect to this facet labeling ($i=1,2$).
If we give a facet labeling of $P_1\times P_2$ by $F_1,\ldots,F_{m_1+m_2}$ where
\[F_i:=\left\{\begin{array}{l l}
F_{1,j}\times P_2 & (1\leq j\leq m_1) \\
P_1\times F_{2,j-m_1} & (m_1+1\leq j\leq m_1+m_2),
\end{array}\right.\]
then we have the following.

\begin{prop}\label{prop:qt bundle}
Let $(M_i,\pi_i)$ be a quasitoric manifold over $P_i$ acted on by $T_i$ ($i=1,2$) and $p\co M\to M_2$ be a quasitoric $M_1$-bundle over $M_2$.
Then $M$ is a quasitoric manifold over $P_1\times P_2$, which has a characteristic matrix in the form
\[\left(\begin{array}{c c}
\lambda_1 & * \\
0 & \lambda_2
\end{array}\right).\]
Conversely, if a quasitoric manifold $M$ over $P_1\times P_2$ has a characteristic matrix in the above form, then $\lambda_i$ is a characteristic matrix on $P_i$ ($i=1,2$) and $M$ is isomorphic to the total space of a quasitoric $M(\lambda_1)$-bundle over $M(\lambda_2)$.
\end{prop}

We use the following lemma to prove this proposition.

\begin{lem}\label{sublem:qt bundle}
Let $(M_i,\pi_i)$ be a quasitoric manifold over $P_i$ acted on by $T_i$ ($i=1,2$) and $p\co M\to M_2$ be a quasitoric $M_1$-bundle over $M_2$.
Moreover, take $x\in M_2$ and put $M_x:=p^{-1}(x)$, $T':=\pr_2^{-1}((T_2)_x)$ where $(T_2)_{x}$ denotes the isotropy subgroup at $x\in M_2$.
Then the action of $T$ on $M$ restricts to a $T'$-action on the fiber $M_x$, and there exists a homomorphism $\rho\co T' \to T_1$ such that $t\cdot y=\rho(t)\cdot y$ for any $t\in T'$ and $y\in M_x$.
In particular, for each $z\in M_x$, there is a split exact sequence 
\[\xymatrix{
0 \ar[r] & (T_1)_z \ar[r] & T_z \ar[r] & (T_2)_{x} \ar[r] & 0.
}\]
\end{lem}

\begin{proof}
Take a local trivialization $\phi\co U\times M_1\to p^{-1}(U)$ of $p$ around $x$ and define $\varphi\co M_1\to M_x$ by $\varphi(y):=\phi(x,y)$.
Since $p$ is a $T_2$-equivariant fiber bundle with structure group $\frD(M_1)$, there exists a map $\gamma\co T'\to \frD(M_1)$ such that $t\cdot \varphi(y)=\varphi(\gamma(t)(y))$ for any $t\in T'$ and $y\in M_1$, which is clearly a homomorphism.
Moreover, since $T_1$ acts on ${\pi_1}^{-1}(\mathrm{int}\,P_1)$ freely and each $\gamma(t)$ ($t\in T'$) descends to $\id_{P_1}$, there is a unique $s_{q,y,t}\in T_1$ for each $q\in\mathrm{int}\,P_1$, $y\in {\pi_1}^{-1}(q)$, and $t\in T'$ such that $\gamma(t)(y)=s_{q,y,t}\cdot y$.
Since $\gamma(t)$ is $T_1$-equivariant, $s_{q,y,t}$ does not depend on $y$ and therefore we can put $s(q,t):=s_{q,y,t}$.
Moreover, for each $q$, the correspondence $t\mapsto s(q,t)$ gives a homomorphism from $T'$ to $T_1$.
Thus we see that the correspondence $q\mapsto s(q,\,\cdot\,)$ gives a map from $\mathrm{int}\,P_1$ to $\mathrm{Hom}(T',T_1)$, the set of continuous homomorphisms equipped with the compact-open topology.
Since $\mathrm{Hom}(T',T_1)$ is discrete and $\mathrm{int}\,P_1$ is connected, this map is constant.
If we define $\rho$ as the value of this map, then $\gamma(t)(y)=\rho(t)\cdot y$ for any $t\in T'$ and $y\in{\pi_1}^{-1}(\mathrm{int}\,P_1)$.
This identity holds for any $t\in T'$ and $y\in M_1$ since ${\pi_1}^{-1}(\mathrm{int}\,P_1)$ is dense in $M_1$.
Thus we obtain the former part of the lemma.

For each $z\in M_x$, the correspondence $t\mapsto (\rho(t)^{-1},t)$ gives a section of $\pr_2\co T_z\to (T_2)_x$, and $(T_1)_z$ clearly coincides with the kernel of $\pr_2\co T_z \to T_2$.
Thus we obtain the latter part of the lemma.
\end{proof}

\begin{proof}[proof of Proposition \ref{prop:qt bundle}]
First, we show that the $T$-action on $M$ is locally standard.
Recall that any quasitoric manifold has a CW structure without odd dimensional cells, and therefore it is simply connected and its odd degree cohomology vanishes.
By using the long exact sequence of homotopy groups and the Serre spectral sequence associated with $p$, we see that $M$ is also simply connected and has vanishing odd degree cohomology.
Then the local standardness follows immediately from the following theorem of Masuda: a torus manifold with vanishing odd degree cohomology is locally standard (\cite[Theorem 4.1]{M06}).
Here a torus manifold means an even-dimensional closed connected orientable smooth manifold equipped with an effective smooth action of the half-dimensional torus which has at least one fixed point.

Next, we prove that $M/T$ is homeomorphic to $P_1\times P_2$ as a manifold with corners.
It is clear that $p$ descends to a $T_2$-equivariant fiber bundle $\bar{p}\co M/T_1\to M_2$ with fiber $P_1$ and structure group $\{\id_{P_1}\}$ by the definition of $\frD(M_1)$.
Thus we see that there is a $T_2$-equivariant homeomorphism $f\co M/T_1\to P_1\times M_2$, where $T_2$ acts on $P_1\times M_2$ by the action on the second component, such that (i) for any local trivialization $\bar{\phi}\co U\times P_1 \to \bar{p}^{-1}(U)$ of $\bar{p}$ and any $x\in U$ the map $P_1\to P_1$ defined by $q\mapsto \pr_1\circ f\circ\bar{\phi}(x,q)$ is identify, and (ii) $\bar{p}=\pr_2\circ f$.
Then $f$ clearly descends to a homeomorphism $\bar{f}\co M/T\to P_1\times P_2$.
We can prove that $\bar{f}$ preserves corners as follows.
If we take $x\in M$ and denote by $\bar{x}\in M/T$ the equivalence class containing $x$, then $\depth\,\bar{x}=\dim T_x$ by definition.
On the other hand, $\depth\,\bar{f}(\bar{x})=\depth\,\bar{x}_1+\depth\,\bar{x}_2$ where $\bar{f}(\bar{x})=(\bar{x}_1,\bar{x}_2)\in P_1\times P_2$.
If we take a local trivialization $\phi\co U\times M_1\to p^{-1}(U)$ around $p(x)$ and put $(x_2,x_1):=\phi^{-1}(x)$, then $\depth\,\bar{x}_i=\dim (T_i)_{x_i}$ for $i=1,2$ since $\bar{x}_i=\pi_i(x_i)$.
Then we have $\depth\,\bar{x}=\depth\,\bar{f}(\bar{x})$ by Lemma \ref{sublem:qt bundle}.

Next, we consider the characteristic matrix $\lambda$ of $M$.
We denote by $\pi$ the projection $M\to M/T\cong P_1\times P_2$.
Take $x\in M$, a local trivialization $\phi\co U\times M_1\to p^{-1}(U)$ of $p$ around $p(x)$, and put $S_j:=\ell_M(F_j)$ ($j=1,\ldots,m_1+m_2$) where $\ell_M$ denotes the characteristic map associated with $M$.
If $x\in \pi^{-1}(\relint\, F_j)$ for some $j\in \{1,\ldots,m_1\}$, then $\pr_2(S_{j})\subseteq T_2$ fixes $p(x)\in \pi_2^{-1}(\mathrm{int}\,P_2)$ and therefore $S_j\subseteq T_1$.
Since $\varphi$ is $T_1$-equivariant on each fiber, we see $S_j=\ell_{M_1}(F_{1,j})$ for $j=1,\ldots,m_1$.
On the other hand, if $x\in \pi^{-1}(\relint\,F_{j})$ for some $j\in\{m_1+1,\ldots,m_1+m_2\}$, then $(T_1)_x=\{0\}$ and $\pr_2(S_j)$ fixes $p(x)\in\pi_2^{-1}(\relint\,F_{2,j-m_1})$.
Therefore we have $\pr_2(S_j)=\ell_{M_2}(F_{2,j-m_1})$.
Thus we obtain the former part of the proposition.

Finally, we prove the latter part.
It is clear that $\lambda_i$ is a characteristic matrix on $P_i$ ($i=1,2$).
We can assume that $M=M(\lambda):=\calZ_{P_1\times P_2}/T_\lambda$ (see Construction \ref{const:moment-angle mfd}) since they are isomorphic.
Put $m:=m_1+m_2$ and identify $T^{m_1}$ with $T^{m_1}\times\{0\}\subseteq T^m$.
Then $T_{\lambda_1}\subseteq T_\lambda$ and $\overline{T}\!_\lambda:=T_\lambda/T_{\lambda_1}$  is isomorphic to $T_{\lambda_2}$ through the projection to $T^{m_2}$.
If we regard that $T^m$ acts on $\calZ_{P_2}$ through the projection to $T^{m_2}$, then, since $\calZ_{P_1\times P_2}=\calZ_{P_1}\times \calZ_{P_2}$ and $M(\lambda)=(M(\lambda_1)\times \calZ_{P_2})/\overline{T}\!_\lambda$, we have the following commutative diagram.
\[\xymatrix{
M(\lambda_1) \ar[r] & M(\lambda_1)\times \calZ_{P_2} \ar[r]^(0.65){\pr_2} \ar[d] & \calZ_{P_2} \ar[d] \\
& M(\lambda) \ar[r] & M(\lambda_2)
}\]
Here the upper row is a $T^m/T_{\lambda_1}$-equivariant fiber bundle with structure group $\frD(M_1)$ and the right vertical arrow is a principal $\overline{T}\!_\lambda$-bundle.
Thus the proof is completed by Lemma \ref{lem:equivariant bundle}.
\end{proof}

Let $P_i$ be an $n_i$-dimensional simple polytope with a facet labeling $F_{i,1},\ldots,F_{i,m_i}$ ($i=1,\ldots,l$),  put $n:=\sum n_i$, $m:=\sum m_i$, and $P:=P_1\times\cdots\times P_l$.
Given $(\psi_i,\rho_i)\in GL_{n_i}(\bbZ)\times R(P_i)$ ($i=1,\ldots,l$), we define $(\psi_1,\rho_1)\times\cdots\times(\psi_l,\rho_l)\in GL_n(\bbZ)\times R(P)$ as follows:
define $\psi\in GL_n(\bbZ)$ and $\rho\in R(P)$ so that
$$
\psi=\left(
\begin{array}{c c c c}
\psi_1&0 &\cdots & 0\\
0 &\ddots &\ddots&\vdots  \\
\vdots &\ddots &\ddots &0  \\
0 &\cdots &0 &\psi_l 
\end{array}
\right),\quad
\iota(\rho)=\left(
\begin{array}{c c c c}
\iota(\rho_1)&0 &\cdots &0  \\
0 &\ddots &\ddots&\vdots  \\
\vdots &\ddots &\ddots &0  \\
0 &\cdots &0 &\iota(\rho_l) 
\end{array}
\right),
$$
and put $(\psi_1,\rho_1)\times\cdots\times(\psi_l,\rho_l):=(\psi,\rho)$.
Here we label the facets of $P$ by $F_1,\ldots,F_m$ where
$$F_{m_1+\cdots+m_{i-1}+j}:=P_1\times\cdots\times P_{i-1}\times F_{i,j}\times P_{i+1}\times\cdots\times P_l \ (i=1,\ldots,l,\ j=1,\ldots,m_i).$$

\begin{lem}\label{wehomeo on each fiber}
Let $M_i$ be a quasitoric manifold over an $n_i$-dimensional simple polytope $P_i$ ($i=1,\ldots,l$) and consider a sequence
\[\xymatrix{
B_l\ar[r]^(0.44){p_{l-1}}&B_{l-1}\ar[r]^(0.55){p_{l-2}}&\cdots\ar[r]^(0.52){p_2}&B_2 \ar[r]^{p_1} & M_1
}\]
where each $p_i$ is a quasitoric $M_{i+1}$-bundle.
Take a characteristic matrix $\lambda_i$ of $M_i$, $(\psi_i,\rho_i)\in GL_{n_i}(\bbZ)\times R(P_i)$, and put $\lambda'_i:=(\psi_i,\rho_i)\cdot\lambda_i$ for $i=1,\ldots,l$.
Then there exist a sequence
\[\xymatrix{
B'_l\ar[r]^(0.44){p'_{l-1}}&B'_{l-1}\ar[r]^(0.55){p'_{l-2}}&\cdots\ar[r]^(0.52){p'_2}&B'_2\ar[r]^(0.45){p'_1}&M(\lambda'_1),
}\]
where each $p'_i$ is a quasitoric $M(\lambda'_{i+1})$-bundle, and a weakly equivariant homeomorphism from $B_l$ to $B'_l$ represented by $(\psi_l,\rho_l)\times\cdots\times(\psi_1,\rho_1)$.
\end{lem}

\begin{proof}
Put $(\psi'_i,\rho'_i):=(\psi_i,\rho_i)\times(\psi_{i-1},\rho_{i-1})\times\cdots\times(\psi_1,\rho_1)$ ($i=1,\ldots,l$).
By an iterated use of Proposition \ref{prop:qt bundle}, we can take a characteristic matrix $\mu_i$ of $B_i$ ($i=1,\ldots,l$) in the form
$$
\left(
\begin{array}{c c c c}
\lambda_i&* &\cdots &*  \\
0 &\ddots &\ddots&\vdots  \\
\vdots &\ddots &\ddots &*  \\
0 &\cdots &0 &\lambda_1 
\end{array}
\right).
$$
Then, if we put $\mu'_i:=(\psi'_i,\rho'_i)\cdot\mu_i$ ($i=1,\ldots,l$), we see that each $M(\mu'_{i+1})$ is a quasitoric $M(\lambda'_{i+1})$-bundle over $M(\mu'_{i})$ by Proposition \ref{prop:qt bundle}.
The proof is completed by setting $B'_i:=M(\mu'_i)$.
\end{proof}

\subsection{Quasitoric manifolds over $I^n$}

Now we restrict ourselves to the case $P=I^n$.
Hereafter, we always use the facet labeling $F_1,\ldots,F_{2n}$ of $I^n$ defined by
\begin{align*}
F_{i}&:=\{(x_1,\ldots,x_n)\in I^n\,\vert\, x_i=0\},\\
F_{n+i}&:=\{(x_1,\ldots,x_n)\in I^n\,\vert\, x_i=1\}
\end{align*}
for $i=1,\ldots,n$.
Note that this facet labeling is different from the one used in the previous subsection.
We easily see that $\Aut(I^n)$ is generated by $\rho_{i,j}:=(i\ j)(i+n\ j+n)$ and $\rho_k:=(k\ k+n)$ ($i,j,k=1,\ldots,n$) where we regard $\Aut(I^n)$ as a subgroup of the symmetric group $\mathfrak{S}_{2n}$ by using the facet labeling, as in Definition \ref{df:action on Lambda_P}.

\begin{df}
Let $\xi$ be a square matrix of order $n$.
We call $\xi$ a \textit{characteristic square on} $I^n$ if each diagonal component of $\xi$ is equal to $1$ and  $(E_n\:\xi)$ is a characteristic matrix on $I^n$.
We denote by $\Xi_n$ the set of characteristic squares on $I^n$.
For the convenience of notation,
we identify a characteristic square $\xi$ with the characteristic matrix $(E_n\:\xi)$, for example, we write $M(\xi)$ instead of $M((E_n\,\xi))$.
\end{df}

\begin{rem}
Due to Proposition \ref{realization of rep}, any quasitoric manifold over $I^n$ is weakly equivariantly homeomorphic to $M(\xi)$ for some characteristic square $\xi$.
\end{rem}

\begin{df}\label{H^*(xi)}
For a characteristic square $\xi=(\xi_{i,j})$ on $I^n$,
we define a graded ring $H^*(\xi)$, which is canonically isomorphic to $H^*(M(\xi);\bbZ)$, as follows.
Let $\bbZ[X_1,\ldots,X_n]$ be the polynomial ring of which the generators have degree 2,
and $\calI_{\xi}$ be the ideal generated by $u_i(\xi) X_i\,(i=1,\ldots,n)$ where $u_i(\xi):=\sum_{j=1}^n \xi_{i,j} X_j$. 
Then $H^*(\xi)$ is defined by
$$ H^*(\xi):=\bbZ[X_1,\ldots,X_n]/\calI_{\xi} .$$
\end{df}


Next, we consider the bundle-type quasitoric manifolds over $I^n$.

\begin{df}
Let $\xi$ be a characteristic square on $I^n$ and $n_1,\ldots,n_l$ be positive integers summing up to $n$.
Then $\xi$ is called \textit{$(\xi_1,\ldots,\xi_l)$-type} if it is in the form
$$
\left(
\begin{array}{c c c c}
\xi_1&* &\cdots &*  \\
0 &\ddots &\ddots&\vdots  \\
\vdots &\ddots &\ddots &*  \\
0 &\cdots &0 &\xi_{l} 
\end{array}
\right)
$$
where each $\xi_i$ is a characteristic square on $I^{n_i}$.
\end{df}

\begin{lem}\label{bundle-type and type of square}
Let $\xi_i$ be a characteristic square on $I^{n_i}$ ($i=1,\ldots,l$) and consider a sequence
\[\xymatrix{
B_l\ar[r]^(0.44){p_{l-1}}&B_{l-1}\ar[r]^(0.55){p_{l-2}}&\cdots\ar[r]^(0.52){p_2}&B_2\ar[r]^{p_1}&B_1
}\]
where $B_1=M(\xi_1)$.
If each $p_i$ is a quasitoric $M(\xi_{i+1})$-bundle, then $B_l$ is weakly equivariantly homeomorphic to $M(\xi)$ for some $(\xi_l,\ldots,\xi_1)$-type characteristic square $\xi$.
\end{lem}

\begin{proof}
Denote the sum of $n_i$ ($i=1,\ldots,l$) by $n$.
By an iterated use of Proposition \ref{prop:qt bundle}, we see that $B_l$ has a characteristic matrix $\lambda$ in the form
$$
\left(
\begin{array}{c c c c c c c c}
E_{n_l}&*&\cdots&*&\xi_l&* &\cdots &*  \\
0&\ddots&\ddots&\vdots&0 &\ddots &\ddots&\vdots  \\
\vdots&\ddots&\ddots&*&\vdots &\ddots &\ddots &*  \\
0&\cdots&0&E_{n_1}&0 &\cdots &0 &\xi_1
\end{array}
\right).
$$
We denote the left $n\times n$ part of $\lambda$ by $A$.
Then, since $A^{-1}$ is also in the form
$$
\left(
\begin{array}{c c c c}
E_{n_l}&*&\cdots&* \\
0&\ddots&\ddots&\vdots \\
\vdots&\ddots&\ddots&* \\
0&\cdots&0&E_{n_1}
\end{array}
\right),
$$
$A^{-1}\lambda=(E_n\,\xi)$ for some $(\xi_l,\ldots,\xi_1)$-type characteristic square $\xi$.
Thus the proof is completed by Proposition \ref{realization of rep}.
\end{proof}

\section{$(\ccp{2})$-bundle type quasitoric manifolds} 

In this section we give the proof of Theorem \ref{main thm 1}.
We use the facet labeling $F_1,\ldots,F_{2n}$ of $I^n$ defined in Section 2.3.
Let us begin with the following proposition.

\begin{prop}\label{only ccp{2}}
Any quasitoric manifold over $I^2$ is weakly equivariantly homeomorphic to $M(\chi)$ where $\chi$ denotes a characteristic square in the following form:
\begin{equation}\label{ch mat on I^2}
\left(\begin{array}{c c}
1 & a \\
0 & 1
\end{array}\right)\text{ or } \left(\begin{array}{c c}
1 & 2 \\
1 & 1
\end{array}\right).
\end{equation}
\end{prop}

\begin{proof}
Let $\lambda$ be a characteristic matrix of a quasitoric manifold $M$ over $I^2$.
Since $\lambda$ satisfies the nonsingularity condition, there is a pair $(\psi,\rho)\in GL_2(\bbZ)\times (\bbZ/2)^4$ such that
\[\psi\cdot\lambda\cdot\rho=\left(\begin{array}{c c c c}
1 & 0 & 1 & a \\
0 & 1 & b & 1
\end{array}\right)\]
where $a$ and $b$ are integers satisfying $a b=1\pm 1$, i.e. $(a,b)=(0,b),(a,0),\pm(1,2),\pm(2,1)$.
Moreover, by multiplying the first row, the first column and the third column by $-1$ if necessary, we can assume that $b\geq 0$.
If we put $\lambda':=\psi\cdot\lambda\cdot\rho$ and $\sigma:=(1\,2)(3\,4)\in \Aut(I^2)$, then we have
\[\left(\begin{array}{c c}
0 & 1 \\
1 & 0
\end{array}\right)\cdot\lambda'\cdot\iota(\sigma)=\left(\begin{array}{c c c c}
1 & 0 & 1 & b \\
0 & 1 & a & 1
\end{array}\right).\]
Thus the proof is completed by Proposition \ref{realization of rep}.
\end{proof}

Throughout this section, we denote by $\kappa^2$ the latter characteristic square in (\ref{ch mat on I^2}).

\begin{rem}\label{rem:kappa^2}
We easily see that $M(\kappa^2)$ is homeomorphic to $\ccp{2}$.
For instance, since $(E_2 \,\kappa^2)$ is decomposed into a connected sum (see \cite[Section 3.2]{H15}) and $\cp{2}$ is the only one quasitoric manifold over $\Delta^2$ (\cite[Example 1.18]{DJ91}), $M(\kappa^2)$ is homeomorphic to $\ccp{2}$ or $\cp{2}\sharp \overline{\cp{2}}$.
$H^*(\kappa^2)$ (Definition \ref{H^*(xi)}) is not isomorphic to $H^*(\cp{2}\sharp \overline{\cp{2}};\bbZ)$, implying $M(\kappa^2)\cong \ccp{2}$.
Note that, by \cite[Proposition 6.2]{CMS10} of Choi, Masuda, and Suh, the other quasitoric manifolds $M(\chi)$ of Proposition \ref{only ccp{2}} are the Hirzebruch surfaces.
In particular, they are not homeomorphic to $\ccp{2}$.
\end{rem}


Next, we consider the graded ring automorphisms of $H^*(\kappa^2)$.
If we put $x:=X_2$ and $y:=u_2(\kappa^2)=X_1+X_2$ with the notation of Definition \ref{H^*(xi)}, then $$H^*(\kappa^2)=\bbZ[x,y]/(x^2-y^2,x y).$$
Let us denote by $\Aut(H^*(\kappa^2))$ the group of graded ring automorphisms of $H^*(\kappa^2)$ and regard it as a subgroup of $GL_2(\bbZ)$ by identifying an automorphism $\varphi$ with the matrix $A$ defined by
$$ \varphi\left(
\begin{array}{c}
x \\
y
\end{array}
\right)=
A \left(
\begin{array}{c}
x \\
y
\end{array}
\right).$$

\begin{lem}\label{Aut(kappa^2)}
$\Aut(H^*(\kappa^2))=\left\{\left(\begin{array}{c c}
\pm 1 &0  \\
0 &\pm 1
\end{array}
\right)
,\ 
\left(\begin{array}{c c}
0 &\pm 1  \\
\pm 1 &0
\end{array}
\right)\right\}$.
\end{lem}

\begin{proof}
Let $\varphi$ be an automorphism of $H^*(\kappa^2)$ identified with
$$ A=\left(\begin{array}{c c}
a &b  \\
c &d
\end{array}
\right). $$
Since $\varphi(x)\varphi(y)=(a c+b d)y^2=0$ in $H^*(\kappa^2)$, we have $a c=-b d$.
In particular, if $a\neq 0$, we see that $a$ divides $d$ and vice versa, implying $a=\pm d$.
We put $\epsilon:=d/a=\pm 1$, and then obtain $c=-\epsilon b$.
Moreover, since $\varphi$ is an automorphism, $\det A=\epsilon (a^2+b^2)=\pm 1$.
Thus we have $a=\epsilon d=\pm 1$ and $b=c=0$.
Similarly, if we assume $b\neq 0$, then we have $|b|=|c|=1$ and $a=d=0$.
Thus the proof is completed.
\end{proof}

Recall that we put $[m]_\pm:=\{\pm 1,\ldots,\pm m\}$ and define $R_m$ as the group of $(\bbZ/2)$-equivariant permutations of $[m]_\pm$ (Definition \ref{df:iota}).
Let us describe $\rho\in R_m$ by $\rho=(\rho(1),\ldots,\rho(m))$.
Then, if we put $\tau_1:=(-1,-2,-3,-4)$, $\tau_2:=(3,-2,1,4)$ and $\tau_3:=(-1,4,3,2)$, they belong to $R(I^2)$ and there are $\psi_i\in GL_2(\bbZ)$ ($i=1,2,3$) such that $(\psi_i,\tau_i)\cdot(E_2\,\kappa^2)=(E_2\,\kappa^2)$.
By Proposition \ref{realization of rep}, there are weakly equivariant self-homeomorphisms $f_i$ ($i=1,2,3$) of $M(\kappa^2)$ represented by $(\psi_i^{-1},\tau_i^{-1})$, and by Corollary \ref{induced by rep}, we have
$$
f_1^*:=\left(\begin{array}{cc}
-1 &0  \\
0 &-1
\end{array}
\right)
,\,
f_2^*:=\left(\begin{array}{cc}
1 &0  \\
0 &-1
\end{array}
\right)
,\,
f_3^*:=\left(\begin{array}{cc}
0 &-1  \\
-1 &0
\end{array}
\right),$$
where we canonically identify $H^*(M(\kappa^2);\bbZ)$ with $H^*(\kappa^2)$ (note that, in the notation of Corollary \ref{cor:cohomology ring of a quasitoric manifold}, $x=v_4$ and $y=-v_2$).
Since these matrices generate $\Aut(H^*(\kappa^2))$ by Lemma \ref{Aut(kappa^2)}, we have the following.

\begin{lem}\label{realization of Aut(kappa^2)}
Any graded ring automorphism of $H^*(M(\kappa^2);\bbZ)$ is induced by a weakly equivariant self-homeomorphism of $M(\kappa^2)$.
\end{lem}


Then we consider the isomorphisms between the cohomology rings of $(\ccp{2})$-bundle type quasitoric manifolds.

\begin{df}\label{df:calK}
We denote by $\calK_n$ the set of $(\kappa^2,\ldots,\kappa^2)$-type characteristic squares on $I^{2n}$.
For $\xi=(\xi_{i,j})\in\calK_n$ and integers $h,k$ such that $0<h\leq k\leq n$, we define $\xi_{[h,k]}:=(\xi_{i,j})_{i,j=2h-1,\ldots,2k}$, which belongs to $\calK_{k-h+1}$.
We identify $H^*(\xi_{[h,n]})$ with the subring of $H^*(\xi)$ generated by $X_{2h-1},\ldots,X_{2n}$ and $H^*(\xi_{[h,k]})$ with the quotient ring $H^*(\xi_{[h,n]})/(X_{2k+1},\ldots,X_{2n})$, where we use the notation of Definition \ref{H^*(xi)}.
\end{df}

\begin{lem}\label{lem:calK}
Any $n$-stage $(\ccp{2})$-bundle type quasitoric manifold is weakly equivariantly homeomorphic to $M(\xi)$ for some $\xi\in\calK_n$.
\end{lem}

\begin{proof}
By Proposition \ref{only ccp{2}} and Remark \ref{rem:kappa^2}, if a quasitoric manifold $M$ over $I^2$ is homeomorphic to $\ccp{2}$, then $M$ is weakly equivariantly homeomorphic to $M(\kappa^2)$.
Therefore, by Lemma \ref{wehomeo on each fiber}, any ($\ccp{2}$)-bundle type quasitoric manifold is weakly equivariantly homeomorphic to a $\{M(\kappa^2)\}$-bundle type quasitoric manifold.
Then the proof is completed by Lemma \ref{bundle-type and type of square}.
\end{proof}

Thus we see that we only have to consider $M(\xi)$ ($\xi\in\calK_n$) to prove Theorem \ref{main thm 1}.


Next, let $\xi'$ be a characteristic square on $I^n$ ($n> 2$) which is $(\kappa^2,\xi'_0)$-type for some characteristic square $\xi'_0$ on $I^{n-2}$, and denote the first and second rows of $\xi'$ by $(1,2,s_3,\ldots,s_n)$ and $(1,1,t_3,\ldots,t_n)$ respectively.
We continue to use the notation of Definition \ref{H^*(xi)}.


\begin{lem}\label{lem:sublemma, filtration lemma 2-2}
Let $\varphi\co \bbZ[X_1,X_2]\to \bbZ[X_1,\ldots,X_n]$ be a graded ring monomorphism
which maps $\calI_{\kappa^2}$ into $\calI_{\xi'}$.
Additionally, we put $\varphi(X_1)=\sum_{i=1}^n a_i X_i$, $\varphi(X_2)=\sum_{i=1}^n b_i X_i$ and assume that for any prime $p$ the $\bmod\, p$ reductions of $(a_1,\ldots,a_n)$ and $(b_1,\ldots,b_n)$ are linearly independent.
Then either of the following $({\rm i})$ and $({\rm ii})$ holds:
\begin{enumerate}
 \item[(\,i\,)] $a_i=b_i=s_i=t_i=0$ for $i=3,4,\ldots,n$;
 \item[(ii)] $a_1=a_2=b_1=b_2=0$.
\end{enumerate}
\end{lem}

\begin{proof}
We prove the lemma by showing (i) under the assumption $(a_1,a_2,b_1,b_2) \neq (0,0,0,0)$.
Since $\varphi( X_1 ( X_1 + 2 X_2 ) )$ and $\varphi( X_2 ( X_1 + X_2 ) )$ belong to $\calI_{\xi'}$,
we have
\begin{align*}
\varphi( X_1 ( X_1 + 2 X_2 ) ) &= \left( \sum_{i=1}^n a_i X_i \right)\left\{ \sum_{i=1}^n ( a_i + 2 b_i ) X_i \right\} \\
 &\equiv \alpha_1 X_1 \left( X_1 + 2 X_2 + \sum_{j=3}^n s_j X_j \right) + \beta_1 X_2 \left( X_1 + X_2 + \sum_{j=3}^n t_j X_j \right) \bmod W, \\
\varphi( X_2 ( X_1 + X_2 ) ) &= \left( \sum_{i=1}^n b_i X_i \right)\left\{ \sum_{i=1}^n ( a_i + b_i ) X_i \right\} \\
 &\equiv \alpha_2 X_1 \left( X_1 + 2 X_2 + \sum_{j=3}^n s_j X_j \right) + \beta_2 X_2 \left( X_1 + X_2 + \sum_{j=3}^n t_j X_j \right) \bmod W
\end{align*}
for some integers $\alpha_i, \beta_i\, (i=1,2)$, where $W$ denotes the submodule spanned by $\{ X_p X_q \,\vert\, p,q \geq 3 \}$.
Since the coefficients of ${X_1}^2$ in $\varphi( X_1 ( X_1 + 2 X_2 ) )$ and $\varphi( X_2 ( X_1 + X_2 ) )$ are
$a_1(a_1 + 2 b_1)$ and $b_1(a_1 + b_1)$ respectively,
we obtain $\alpha_1=a_1(a_1 + 2 b_1)$ and $\alpha_2=b_1(a_1 + b_1)$.
Similarly, we see $\beta_1=a_2(a_2+2 b_2)$ and $\beta_2=b_2(a_2+b_2)$.
Thus we obtain the following equations.
\begin{eqnarray*}
  a_1(a_2+2 b_2) + a_2(a_1+2 b_1) & = & 2 a_1(a_1+2 b_1) + a_2(a_2+2 b_2) \\
  a_1(a_i+2 b_i) + a_i(a_1+2 b_1) & = & a_1(a_1+2 b_1) s_i\quad(i\geq 3) \\
  a_2(a_i+2 b_i) + a_i(a_2+2 b_2) & = & a_2(a_2+2 b_2) t_i\quad(i\geq 3) \\
  b_1(a_2+b_2) + b_2(a_1+b_1) & = & 2 b_1(a_1+b_1) + b_2(a_2+b_2) \\
  b_1(a_i+b_i) + b_i(a_1+b_1) & = & b_1(a_1+b_1) s_i\quad(i\geq 3) \\
  b_2(a_i+b_i) + b_i(a_2+b_2) & = & b_2(a_2+b_2) t_i\quad(i\geq 3) 
\end{eqnarray*}
For the convenience, we rewrite these equations as follows.
\begin{eqnarray}\label{eq:1}
  (a_1-a_2)(a_2+2 b_2) & = & (2 a_1-a_2)(a_1+2 b_1) \\ \label{eq:2}
  a_1(a_i+2 b_i) & = & (s_i a_1 -a_i)(a_1+2 b_1)\quad(i\geq 3) \\ \label{eq:3}
  a_2(a_i+2 b_i) & = & (t_i a_2-a_i)(a_2+2 b_2)\quad(i\geq 3) \\ \label{eq:4}
  (b_1-b_2)(a_2+b_2) & = & (2 b_1-b_2)(a_1+b_1) \\ \label{eq:5}
  b_1(a_i+b_i) & = & (s_i b_1-b_i)(a_1+b_1)\quad(i\geq 3) \\ \label{eq:6}
  b_2(a_i+b_i) & = & (t_i b_2-b_i)(a_2+b_2)\quad(i\geq 3) 
\end{eqnarray}

First, we assume that all of $a_1-a_2$, $2 a_1-a_2$, $b_1-b_2$ and $2 b_1-b_2$ are non-zero.
Let $k>0$ be the greatest common divisor of $a_1-a_2$ and $2 a_1-a_2$,
and $l>0$ be that of $b_1-b_2$ and $2 b_1-b_2$.
Suppose that $r$ divides $a_1+2 b_1$ and $a_2+2 b_2$.
If we assume that $r$ does not divide $k$, then there is a prime number $r'$ which divides $r/(k,r)$ but does not divide $k/(k,r)$,
where $(k,r)$ means the greatest common divisor.
Then, by (\ref{eq:2}) and (\ref{eq:3}), $r'$ divides $a_i+2 b_i$ for $i=1,2,\ldots,n$,
but it contradicts the assumption (b).
Thus we see that any common divisor of $a_1+2 b_1$ and $a_2+2 b_2$ divides $k$.
In particular, $a_1+2 b_1, a_2+2 b_2 \neq 0$.
Similarly, we shall show that any common divisor of $a_1+b_1$ and $a_2+b_2$ divides $l$.

Let $p>0$ be the greatest common divisor of $a_1+2 b_1$ and $a_2+2 b_2$,
and $q>0$ be that of $a_1+b_1$ and $a_2+b_2$.
Since $\frac{a_1-a_2}{k}$ and $\frac{2 a_1-a_2}{k}$ (resp. $\frac{a_1+2 b_1}{p}$ and $\frac{a_2+2 b_2}{p}$) are prime to each other,
we obtain
$$ \cfrac{\; \cfrac{a_1-a_2}{k}\; }{\; \cfrac{a_1+2 b_1}{p}\; } = \cfrac{\; \cfrac{2 a_1-a_2}{k}\; }{\; \cfrac{a_2+2 b_2}{p}\; } = \pm 1 $$
from (\ref{eq:1}).
It can be written as
\begin{equation}\label{eq:7}
\frac{a_1-a_2}{a_1+2 b_1} = \frac{2 a_1-a_2}{a_2+2 b_2} = \pm \frac{k}{p} \in \bbZ .
\end{equation}
Similarly, we obtain
\begin{equation}\label{eq:8}
\frac{b_1-b_2}{a_1+b_1} = \frac{2 b_1-b_2}{a_2+b_2} = \pm \frac{l}{q} \in \bbZ .
\end{equation}

Define
$$ k':=\frac{a_1-a_2}{a_1+2 b_1},\ l':=\frac{b_1-b_2}{a_1+b_1}\in \bbZ .$$
Then, from (\ref{eq:7}) and (\ref{eq:8}), we have the following equations.
\begin{eqnarray}\label{eq:9}
  \left\{
    \begin{array}{l}
      (1-k')a_1-a_2-2 k' b_1 = 0 \\
      2 a_1+(-1-k')a_2-2 k' b_2 = 0 \\
      -l' a_1+(1-l')b_1-b_2 = 0 \\
      -l' a_2+2 b_1+(-1-l')b_2 = 0
    \end{array}
  \right.
\end{eqnarray}
Since we assume $a_i,b_i\neq 0\,(i=1,2)$, the determinant of the matrix
$$
A:=
\left(
\begin{array}{c c c c}
1-k' &-1 &-2 k' &0  \\
2 &-1-k' &0 &-2 k'  \\
-l' &0 &1-l' &-1  \\
0 &-l' &2 &-1-l' 
\end{array}
\right)
$$
equals $0$.
Therefore, since $\det A=(k'+l')^2+(k' l'+1)^2$, we obtain $(k',l')=(1,-1),(-1,1)$.

If $(k',l')=(1,-1)$, we have
$$ a_1=b_2-2 b_1,\ a_2=-2 b_1 $$
from (\ref{eq:9}).
If $b_1=0$, we easily obtain $a_j=b_j=s_j=t_j=0\,(j\geq 2)$ from (\ref{eq:2}), (\ref{eq:3}), (\ref{eq:5}) and (\ref{eq:6}).
On the other hand, if we assume $b_2=0$, we similarly obtain $a_j=b_j=s_j=t_j=0\,(j\geq 2)$
(but this contradicts the assumption (b) since $(a_1,\ldots,a_n)\equiv 0 \bmod 2$).

Then we can assume that $b_1,b_2\neq 0$.
Putting $b'_i:=b_i/l$ ($i=1,2$), we obtain the following equations from (\ref{eq:2}), (\ref{eq:3}), (\ref{eq:5}) and (\ref{eq:6}).
\begin{eqnarray}\label{eq:10}
 (b'_2-2 b'_1)(a_i+2 b_i) & = & \{s_i (b_2-2 b_1) -a_i\}b'_2\quad(i\geq 3) \\ \label{eq:11}
 b'_1(a_i+2 b_i) & = & (-2 t_i b_1-a_i)(b'_1-b'_2)\quad(i\geq 3) \\ \label{eq:12}
 b'_1(a_i+b_i) & = & (s_i b_1-b_i)(b'_2-b'_1)\quad(i\geq 3) \\ \label{eq:13}
 b'_2(a_i+b_i) & = & (t_i b_2-b_i)(-2 b'_1+b'_2)\quad(i\geq 3) 
\end{eqnarray}
If $b_2$ is odd, $b'_2-2 b'_1$ and $b'_2$ (resp. $b'_1$ and $b'_1-b'_2$) are prime to each other,
and hence we obtain the following.
\begin{eqnarray}\label{eq:14}
 \frac{a_i+2 b_i}{b'_2} & = & s_i l - \frac{a_i}{b'_2-2 b'_1}\in\bbZ\quad(i\geq 3) \\ \label{eq:15}
 \frac{a_i+2 b_i}{b'_1-b'_2} & = & -2 t_i l - \frac{a_i}{b'_1}\in\bbZ\quad(i\geq 3) \\ \label{eq:16}
 \frac{a_i+b_i}{b'_2-b'_1} & = & s_i l - \frac{b_i}{b'_1}\in\bbZ\quad(i\geq 3) \\ \label{eq:17}
 \frac{a_i+b_i}{b'_2-2 b'_1} & = & t_i l - \frac{b_i}{b'_2}\in\bbZ\quad(i\geq 3) 
\end{eqnarray}
In particular, $b'_1$ divides $a_i$ and $b_i$ for $i=3,4,\ldots,n$.
Putting $a'_i:=a_i/b'_1$ and $b'_i:=b_i/b'_1$ ($i\geq 3$), from (\ref{eq:14}) and (\ref{eq:16}) (resp. (\ref{eq:15}) and (\ref{eq:17})),
we obtain
\begin{eqnarray}\label{eq:18}
 \{(a'_i+2 b'_i)(b'_2-2 b'_1) + a'_i b'_2\}(b'_2-b'_1)b'_1 & = & \{(a'_i+b'_i)b'_1 + b'_i(b'_2-b'_1)\}(b'_2-2 b'_1)b'_2, \\ \label{eq:19}
 \{(a'_i+2 b'_i)b'_1 + a'_i(b'_1-b'_2)\}(b'_2-2 b'_1)b'_2 & = & -2\{(a'_i+b'_i)'_2 + b'_i(b'_2-2 b'_1)\}b'_1(b'_1-b'_2)
\end{eqnarray}
for $i=3,\ldots,n$.
Since $b'_1$ is prime to $b'_2-2 b'_1$, $b'_2$ and $b'_2-b'_1$ respectively, we see that $b'_1$ divides $b'_i\,(i=3,\ldots,n)$ from (\ref{eq:18})
and divides $a'_i\,(i=3,\ldots,n)$ from (\ref{eq:19}).
Repeating this procedure, we see that any power of $b'_1$ divides $a_i,b_i\,(i=3,\ldots,n)$.
By similar arguments, we can show that any powers of $b'_2$, $2 b'_1-b'_2$ and $b'_1-b'_2$ divide $a_i,b_i\,(i=3,\ldots,n)$.
$b'_1$, $b'_2$, $2 b'_1-b'_2$ and $b'_1-b'_2$ cannot be $\pm 1$ simultaneously, and hence $a_i,b_i=0\,(i=3,\ldots,n)$.
Then we obtain $s_i,t_i=0\,(i=3,\ldots,n)$ from (\ref{eq:14}) and (\ref{eq:15}).
 
Otherwise, if $b'_2$ is even (and hence $b'_1$ is odd), put $b''_2:=b'_2/2$.
Then we have the following.
\begin{eqnarray*}
 (b''_2-b'_1)(a_i+2 b_i) & = & \{s_i (b_2-2 b_1) -a_i\}b''_2\quad(i\geq 3) \\
 b'_1(a_i+2 b_i) & = & (-2 t_i b_1-a_i)(b'_1-2 b''_2)\quad(i\geq 3) \\
 b'_1(a_i+b_i) & = & (s_i b_1-b_i)(2 b''_2-b'_1)\quad(i\geq 3) \\
 b''_2(a_i+b_i) & = & (t_i b_2-b_i)(b''_2-b'_1)\quad(i\geq 3) 
\end{eqnarray*}
By an argument similar to above, we obtain $a_i,b_i,s_i,t_i=0\,(i=3,\ldots,n)$ again.
We shall obtain (i) in the same way if $(k',l')=(-1,1)$. 

Finally, we consider the case where at least one of $a_1-a_2$, $2 a_1-a_2$, $b_1-b_2$ and $2 b_1-b_2$ equals zero.
Note that ${X_2}^2$ and $X_p X_q$ ($p=1,2$, $q=3,\ldots,n$) form a basis of $H^4(\xi')/W$.
Then, by considering the equation $\varphi(X_2(X_1+X_2))=0$ in $H^4(\xi')/W$, we see that $(a_1,a_2)=(0,0)$ implies $(b_1,b_2)=(0,0)$ and vice versa.
Since we assume $(a_1,a_2,b_1,b_2) \neq (0,0,0,0)$ as mentioned at the beginning of this proof, we have $(a_1,a_2)\neq (0,0)$ and $(b_1,b_2)\neq(0,0)$.

If $a_1-a_2=0$, then we have $0=a_1(a_1+2b_1)$ by (\ref{eq:1}), which implies $a_1+2b_1=0$ since $(a_1,a_2)\neq 0$.
Similarly, we obtain $a_i+2b_i=0$ ($i\geq 3$) by (\ref{eq:2}), but this contradicts the assumption of linear independence since $(a_1,\ldots,a_n)\equiv 0 \bmod 2$.

If $2a_1-a_2=0$, then we have $a_2+2b_2=0$ by (\ref{eq:1}), which implies $b_2=-a_1$.
By (\ref{eq:3}) and (\ref{eq:2}), we have $a_i+2b_i=0$ ($i\geq 3$) and then $s_i a_1-a_i=0$ ($i\geq 3$).
Moreover, $a_i+t_i b_2=0$ by (\ref{eq:6}).
We have $b_1(a_1+b_1)=0$ by (\ref{eq:4}).
If $b_1=0$, then (\ref{eq:5}) implies $b_i=0$ ($i\geq 3$), and therefore we obtain $a_i=s_i=t_i$ ($i\geq 3$) since both $a_1$ and $b_2$ are non-zero.
If $a_1+b_1=0$, then (\ref{eq:5}) implies $a_i+b_i=0$ ($i\geq 3$), and therefore $a_i=b_i=s_i=t_i=0$ ($i\geq 3$).
We can show $a_i=b_i=s_i=t_i=0$ ($i\geq 3$) similarly in the other two cases.
\end{proof}


For $\rho\in \mathfrak{S}_m$, the symmetric group, and a positive integer $k$, we define $\rho[k]\in \mathfrak{S}_{k m}$ so that $\rho[k](i k-j)=\rho(i)\cdot k-j$ for $i=1,\ldots,m$ and $j=0,\ldots,k-1$.

\begin{lem}\label{lem:permutation}
Let $\xi'\in\calK_n$, $s_{i,j}$ be its $(i,j)$-th entry, and assume that there exist two integers $p$, $q$ which satisfy the following:
\begin{enumerate}
 \item[(a)] $0<p<q\leq n$;
 \item[(b)] $s_{i,j}=0$ if $i=2p-1,2p$ and $j=2p+1,\ldots,2q$.
\end{enumerate}
Then there exist $\eta'\in\calK_n$ and a weakly equivariant homeomorphism $f\co M(\eta')\rightarrow M(\xi')$ such that $f^*(X_i)=X_{\sigma[2](i)}$ ($i=1,\ldots,2n$) where $\sigma$ denotes the cyclic permutation $(q\ q-1\ \cdots\ p)$.
\end{lem}

\begin{proof}
Recall that we put $\rho_{i,j}:=(i\ j)(i+n\ j+n)\in \Aut(I^n)$.
If we put $\omega_k:=\rho_{k,k+1}[2]$ and $\omega_{p,q}:=\omega_{q-1}\circ\cdots\circ\omega_p$, then $\omega_{p,q}\in \Aut(I^{2n})$.
Additionally, we put $\psi:=\iota(\sigma[2])^{-1}\in GL_{2n}(\bbZ)$ and $\sigma_k:=(k\,k+1)\in \mathfrak{S}_n$.
Since $\iota$ is an antihomomorphism,
$$(\psi,\omega_{p,q})\cdot(E_{2n}\,\xi')= \iota(\sigma_{q-1}[2])^{-1}\cdots\iota(\sigma_p[2])^{-1}\cdot(E_{2n}\,\xi')\cdot\iota(\omega_p)\cdots\iota(\omega_{q-1}).$$
We can easily check that $\iota(\sigma_p[2])^{-1}\cdot(E_{2n}\,\xi')\cdot\iota(\omega_p)=(E_{2n}\,\xi'')$ where $\xi''=(s'_{i,j})\in\calK_n$ and it satisfies the following:
$s'_{i,j}=0$ if $i=2p+1,2p+2$ and $j=2p+3,\ldots,2q$.
By induction, we see that $(\psi,\omega_{p,q})\cdot(E_{2n}\,\xi')=(E_{2n}\,\eta')$ for some $\eta'\in\calK_n$.
Then the proof is completed by Proposition \ref{realization of rep} and Corollary \ref{induced by rep}.
\end{proof}

\begin{lem}\label{lem:filtration lemma 2-2}
Let $\xi,\xi'\in\calK_n$ and $\varphi \co H^*(\xi) \rightarrow H^*(\xi')$ be a graded ring isomorphism.
Then there exist $\eta'\in\calK_n$ and a weakly equivariant homeomorphism $f\co M(\eta')\rightarrow M(\xi')$ such that $f^* \circ \varphi$ preserves the ideal $(X_{2i-1},\ldots,X_{2n})$ for each $i=1,\ldots,n$.
\end{lem}

\begin{proof}
First, by Lemma \ref{lem:sublemma, filtration lemma 2-2} and Lemma \ref{lem:permutation}, there are $\eta''\in\calK_n$ and a weakly equivariant homeomorphism $f'\co M(\eta'')\to M(\xi')$ such that $f'^*\circ\varphi$ preserves the ideal $(X_{2n-1},X_{2n})$.
Therefore, without loss of generality, we can assume that $\varphi$ preserves $(X_{2n-1},X_{2n})$.

We prove the lemma by induction on $n$.
The lemma is trivial if $n=1$.
Suppose that the lemma holds for $n-1$.
If $\varphi$ preserves the ideal $(X_{2n-1},X_{2n})$, then it descends to a graded ring isomorphism $\bar{\varphi}\co H^*(\xi_{[1,n-1]})\to H^*(\xi'_{[1,n-1]})$ (see Definition \ref{df:calK}).
By the induction hypothesis, there are $\eta'_0\in\calK_{n-1}$ and a weakly equivariant homeomorphism $f_0\co M(\eta'_0)\rightarrow M(\xi'_{[1,n-1]})$ such that $f_0^* \circ\bar{\varphi}$ preserves the ideal $(X_{2i-1},\ldots,X_{2n-2})$ for each $i=1,\ldots,n-1$.
Let $(\psi_0,\rho_0)$ be the representation of $f_0$ and put $(\psi,\rho):=(\psi_0,\rho_0)\times (E_2,e)\in GL_{2n}(\bbZ)\times R(I^{2n})$, where $e$ denotes the identity element of $R(I^2)$.
The product $(\psi_0,\rho_0)\times (E_2,e)$ is defined before Lemma \ref{wehomeo on each fiber}, but we should note that now we use a different facet labeling.
If we define a characteristic square $\eta'$ on $I^{2n}$ so that $(\psi,\rho)\cdot (E_{2n}\,\eta')=(E_{2n}\,\xi')$, then there exists a weakly equivariant homeomorphism $f\co M(\eta')\rightarrow M(\xi')$ represented by $(\psi,\rho)$.
We can easily check that $\eta'\in\calK_n$ and  $f^* \circ \varphi$ preserves the ideal $(X_{2i-1},\ldots,X_{2n})$ for each $i=1,\ldots,n$.
%
%
\end{proof}

\begin{cor}\label{1-I}
Let $\xi,\xi'\in\calK_n$ and $\varphi \co H^*(\xi) \rightarrow H^*(\xi')$ be a graded ring isomorphism.
Then there exist $\eta'\in\calK_n$ and a weakly equivariant homeomorphism $f\co M(\eta')\rightarrow M(\xi')$ such that 
$$ f^* \circ \varphi\left(
\begin{array}{c}
X_1 \\
\vdots \\
X_{2n}
\end{array}
\right)=
\left(
\begin{array}{c c c c}
E_2 &* &\cdots &*  \\
0 &\ddots &\ddots&\vdots  \\
\vdots &\ddots &\ddots &*  \\
0 &\cdots &0 &E_2
\end{array}
\right)
\left(
\begin{array}{c}
X_1 \\
\vdots \\
X_{2n}
\end{array}
\right).$$
\end{cor}

\begin{proof}
By Lemma \ref{lem:filtration lemma 2-2}, there exist $\eta''\in\calK_n$ and a weakly equivariant homeomorphism $f_1\co M(\eta'')\rightarrow M(\xi')$ such that 
$$ f_1^* \circ \varphi\left(
\begin{array}{c}
X_1 \\
\vdots \\
X_{2n}
\end{array}
\right)=
\left(
\begin{array}{c c c c}
\alpha_1 &* &\cdots &*  \\
0 &\ddots &\ddots&\vdots  \\
\vdots &\ddots &\ddots &*  \\
0 &\cdots &0 &\alpha_n
\end{array}
\right)
\left(
\begin{array}{c}
X_1 \\
\vdots \\
X_{2n}
\end{array}
\right),$$
where each $\alpha_i$ ($i=1,\ldots,n$) gives an automorphism of $H^*(\kappa^2)$ since $\xi_{[i,i]}=\xi'_{[i,i]}=\kappa^2$.
Moreover, by Lemma \ref{realization of Aut(kappa^2)}, there is a weakly equivariant self-homeomorphism $h_i\co M(\kappa^2)\to M(\kappa^2)$ such that $h_i^*=\alpha_i$.
Take the representation $(\psi_i,\rho_i)$ of $h_i$ ($i=1,\ldots,n$) and put $(\psi,\rho):=(\psi_1,\rho_1)\times\cdots\times(\psi_n,\rho_n)$.
If we define a characteristic square $\eta'$ on $I^{2n}$ so that $(E_{2n}\,\eta')=(\psi,\rho)\cdot(E_{2n}\,\eta'')$ and take a weakly equivariant homeomorphism $f_2\co M(\eta')\to M(\eta'')$ represented by $(\psi^{-1},\rho^{-1})$, then $\eta'\in\calK_n$ and
$$ f_2^* \left(
\begin{array}{c}
X_1 \\
\vdots \\
X_{2n}
\end{array}
\right)=
\left(
\begin{array}{c c c c}
\alpha_1^{-1} &0 &\cdots &*  \\
0 &\ddots &\ddots&\vdots  \\
\vdots &\ddots &\ddots &0  \\
0 &\cdots &0 &\alpha_n^{-1}
\end{array}
\right)
\left(
\begin{array}{c}
X_1 \\
\vdots \\
X_{2n}
\end{array}
\right).$$
If we put $f:=f_1\circ f_2$, then it satisfies the condition of the lemma.
\end{proof}

%

\begin{lem}\label{lemma for 1-II}
Let $\xi_0$ be a characteristic square on $I^{n-2}$ ($n\geq 3$), $\xi,\xi'$ be two $(\kappa^2,\xi_0)$-type characteristic squares on $I^n$, and $\varphi \co H^*(\xi) \rightarrow H^*(\xi')$ be a graded ring isomorphism such that 
$$ \varphi\left(
\begin{array}{c}
X_1 \\
\vdots \\
X_n
\end{array}
\right)=
\left(
\begin{array}{c c}
E_2 &A  \\
0 &E_{n-2}
\end{array}
\right)
\left(
\begin{array}{c}
X_1 \\
\vdots \\
X_n
\end{array}
\right)$$
where $A$ denotes some $(2\times (n-2))$-matrix of integers.
Then we have $A=0$ and $\xi=\xi'$.
\end{lem}

\begin{proof}
We denote the first rows of $A$, $\xi$, $\xi'$ by $(a_3,\ldots,a_n)$, $(1,2,s_3,\ldots,s_n)$, $(1,2,s'_3,\ldots,s'_n)$ respectively.
Similarly, we denote their second rows by $(b_3,\ldots,b_n)$, $(1,1,t_3,\ldots,t_n)$, $(1,1,t'_3,\ldots,t'_n)$.
Then, in $H^*(\xi')$, we have
\begin{align*}
& \varphi(X_1(X_1+2 X_2+s_3 X_3+\cdots+s_n X_n)) \\
& = ( X_1+a_3 X_3+\cdots+a_n X_n )\{ X_1+2 X_2+ (a_3+2 b_3+s_3) X_3+\cdots+ (a_n+2 b_n+s_n) X_n \} \\
& = X_1 \{ (2 a_3+2 b_3+s_3-s'_3)X_3+\cdots+(2 a_n+2 b_n+s_n-s'_n)X_n \} + \\
& \quad \quad \quad \quad  \quad \quad 2 X_2\{ a_3 X_3+\cdots+a_n X_n \} + (\text{a polynomial in } X_3,\ldots,X_n)=0
\end{align*}
and
\begin{align*}
& \varphi(X_2(X_1+X_2+t_3 X_3+\cdots+t_n X_n)) \\
& = ( X_2+b_3 X_3+\cdots+b_n X_n )\{ X_1+X_2+ (a_3+b_3+t_3) X_3+\cdots+ (a_n+b_n+t_n) X_n \} \\
& = X_2 \{ (a_3+2 b_3+t_3-t'_3)X_3+\cdots+(a_n+2 b_n+t_n-t'_n)X_n \} + \\
& \quad \quad \quad \quad  \quad \quad X_1\{ b_3 X_3+\cdots+b_n X_n \} + (\text{a polynomial in } X_3,\ldots,X_n) =0.
\end{align*}
If we define $W$ as the submodule of $H^4(\xi')$ generated by $X_p X_q$ ($p,q\geq 3$), then ${X_2}^2$ and $X_i X_j$ ($i=1,2,\,j=3,\ldots,n$) form a basis of $H^4(\xi')/W$. 
Therefore we obtain $a_i=b_i=s_i-s'_i=t_i-t'_i=0$, i.e. $A=0$ and $\xi=\xi'$.
\end{proof}

\begin{cor}\label{1-II}
Let $\xi,\xi'\in\calK_n$ and $\varphi \co H^*(\xi) \rightarrow H^*(\xi')$ be a graded ring isomorphism such that 
$$ \varphi\left(
\begin{array}{c}
X_1 \\
\vdots \\
X_{2n}
\end{array}
\right)=
\left(
\begin{array}{c c c c}
E_2 &A_{1,2} &\cdots &A_{1,n}  \\
0 &\ddots &\ddots&\vdots  \\
\vdots &\ddots &\ddots &A_{n-1,n}  \\
0 &\cdots &0 &E_2
\end{array}
\right)
\left(
\begin{array}{c}
X_1 \\
\vdots \\
X_{2n}
\end{array}
\right).$$
Then $A_{i,j}=0$ ($1\leq i<j\leq n$) and $\xi=\xi'$.
In particular, $\varphi$ is induced by $\id_{M(\xi)}$.
\end{cor}

\begin{proof}
We prove the corollary by induction on $n$.
If $n=2$, the corollary is immediate from Lemma \ref{lemma for 1-II}.
Suppose that the corollary holds for $n-1$.
Since $\varphi$ restricts to a graded ring isomorphism from $H^*(\xi_{[2,n]})$ to $H^*(\xi'_{[2,n]})$, by the induction hypothesis, we obtain $A_{i,j}=0$ for $2\leq i<j\leq n$ and $\xi_{[2,n]}=\xi'_{[2,n]}$.
Then we have $A_{1,j}=0$ ($1<j\leq n$) and $\xi=\xi'$ by Lemma \ref{lemma for 1-II}.
\end{proof}

%

Then the following theorem is immediate from Corollary \ref{1-I} and Corollary \ref{1-II}.

\begin{thm}\label{main thm 1'}
Let $\xi,\xi'\in\calK_n$ and $\varphi \co H^*(\xi) \rightarrow H^*(\xi')$ be a graded ring isomorphism.
Then there exists a weakly equivariant homeomorphism $f\co M(\xi')\to M(\xi)$ such that $\varphi=f^*$.
\end{thm}

By Lemma \ref{lem:calK} and Theorem \ref{main thm 1'}, we obtain Theorem \ref{main thm 1}.

\section{Computation of $\calM^{\rm weh}_{I^3}$}

Toward the proof of Theorem \ref{main thm 2}, in this section, we list all the quasitoric manifolds over $I^3$ up to weakly equivariant homeomorphism.
We denote by $\calM^{\rm weh}_{I^3}$ the set of weakly equivariant homeomorphism classes of quasitoric manifolds over $I^3$, as in Corollary \ref{classification up to wehomeo}.

\begin{notation}
To compute $\calM_{I^3}^{\rm weh}$, we use the following notations.
Recall that we denote by $\Xi_3$ the set of characteristic squares on $I^3$.
\begin{itemize}
\item  We denote by $\phi \co \Xi_3\to \calM^{\rm weh}_{I^3}$ the surjection given by $\xi\mapsto M(\xi)$.
 \item For $V_1,V_2,V_3\subseteq \bbZ^2$, we define
 $$\Xi(V_1,V_2,V_3):=\Set{
\left(
\begin{array}{ccc}
1 &x_1 &x_2  \\
y_1 &1 &x_3  \\
y_2 &y_3 &1 
\end{array}
\right)
\in \Xi_3 | \left(
\begin{array}{c}
x_i  \\
y_i
\end{array}
\right)\in V_i\text{ for } i=1,2,3
}.$$
 \item We put $P_+:=\Set{\left(
\begin{array}{c}
k  \\
0
\end{array}
\right)| k\in\bbZ}$,
$P_-:=\Set{\left(
\begin{array}{c}
0  \\
k
\end{array}
\right)| k\in\bbZ}$,
$N_+:=\left\{\pm\left(
\begin{array}{c}
2  \\
1
\end{array}
\right)\right\}$,
$N_-:=\left\{\pm\left(
\begin{array}{c}
1  \\
2
\end{array}
\right)\right\}$,
$C_0 := P_+ \cup P_-$, and $C_2:=N_+ \cup N_-$.
\item We put $C_{\epsilon_1,\epsilon_2,\epsilon_3}:=\Xi(C_{\epsilon_1},C_{\epsilon_2},C_{\epsilon_3})$ 
for $(\epsilon_1,\epsilon_2,\epsilon_3)\in \{0,2\}^3$.
\item We define $\sigma_i,\tau_i \in \Aut (I^3)$ ($i=1,2,3$) by
$$ \sigma_1:=(1\ 2)(4\ 5),\ \sigma_2:=(1\ 3)(4\ 6),\ \sigma_3:=(2\ 3)(5\ 6),\ \tau_i:=(i\ i+3)$$
and $\delta_i\in GL_3(\bbZ)$ ($i=1,2,3$) by 
$$ \delta_1:=\left(\begin{array}{ccc}0&1&0\\1&0&0\\0&0&1\end{array}\right),\ 
\delta_2:=\left(\begin{array}{ccc}0&0&1\\0&1&0\\1&0&0\end{array}\right),\ 
\delta_3:=\left(\begin{array}{ccc}1&0&0\\0&0&1\\0&1&0\end{array}\right).$$
Additionally, we take $\mu_i\in(\bbZ/2)^6$ ($i=1,2,3$) such that the $i$-th and ($i+3$)-th components are $-1$ and the other components are $1$, and $\nu_i\in GL_3(\bbZ)$ ($i=1,2,3$) which acts on $\bbZ^3$ by multiplication by $-1$ on the $i$-th component.
\end{itemize}
\end{notation}

\begin{rem}
Since $(\delta_i,\sigma_i)\cdot(E_3\, \xi)$ and $(\nu_i,\mu_i)\cdot(E_3\, \xi)$ (see Definition \ref{df:action on Lambda_P}) are in the form $(E_3\ \xi')$ for each $\xi\in\Xi_3$ and $i=1,2,3$, we can regard that $(\delta_i,\sigma_i)$ and $(\nu_i,\mu_i)$ act on $\Xi_3$.
\end{rem}

\begin{lem}\label{lem:first restriction}
The restriction of $\phi$ to $C_{0,0,0}\cup C_{0,0,2}\cup C_{0,2,2}\cup C_{2,2,2}$ is surjective.
\end{lem}

\begin{proof}
Let $\xi$ be a characteristic square on $I^3$ and write
$$\xi=\left(
\begin{array}{ccc}
1 &x_1 &x_2  \\
y_1 &1 &x_3  \\
y_2 &y_3 &1 
\end{array}
\right).$$
From the nonsingularity condition, $1-x_i y_i=\pm 1$ ($i=1,2,3$).
This implies that each ${}^t(x_i,y_i)$ belongs to $C_0$ or $C_2$.
Therefore we obtain $$\Xi_3=\bigcup_{\epsilon_1,\epsilon_2,\epsilon_3\in\{0,2\}}C_{\epsilon_1,\epsilon_2,\epsilon_3}.$$
Since $(\delta_1,\sigma_1)\cdot C_{\epsilon_1,\epsilon_2,\epsilon_3} =C_{\epsilon_1,\epsilon_3,\epsilon_2}$, we have $\phi(C_{\epsilon_1,\epsilon_2,\epsilon_3})=\phi(C_{\epsilon_1,\epsilon_3,\epsilon_2})$.
Similarly, we have $\phi(C_{\epsilon_1,\epsilon_2,\epsilon_3})=\phi(C_{\epsilon_3,\epsilon_2,\epsilon_1})$ and $\phi(C_{\epsilon_1,\epsilon_2,\epsilon_3})=\phi(C_{\epsilon_2,\epsilon_1,\epsilon_3})$.
Hence we see that
$$ \calM_{I^3}^{\rm weh}=\bigcup_{\epsilon_1,\epsilon_2,\epsilon_3\in\{0,2\}}\phi(C_{\epsilon_1,\epsilon_2,\epsilon_3})
=\phi(C_{0,0,0})\cup \phi(C_{0,0,2})\cup \phi(C_{0,2,2})\cup \phi(C_{2,2,2}) .$$
Thus we obtain the lemma.
\end{proof}

Let us put $P_{s_1,s_2,s_3}:=\Xi(P_{s_1},P_{s_2},P_{s_3})$
where $s_i\in\{+,-\}$ ($i=1,2,3$).
Then we have $$C_{0,0,0}=\bigcup_{s_1,s_2,s_3\in \{+,-\}} P_{s_1,s_2,s_3}. $$
Moreover, $(\delta_i,\sigma_i)$ ($i=1,2,3$) act as follows.
\[
\xymatrix {
P_{+,+,+} \ar[1,0]_{(\delta_3,\sigma_3)} \ar[1,2]^(0.6){(\delta_2,\sigma_2)} \ar[0,2]^{(\delta_1,\sigma_1)} && P_{-,+,+} \ar[1,2]^(0.6){(\delta_3,\sigma_3)} \ar[0,2]^{(\delta_2,\sigma_2)} \ar[0,-2]
&& P_{-,-,+} \ar[0,-2]
& P_{+,-,+} \ar[1,0]^{(\delta_2,\sigma_2)} \\
P_{+,+,-} \ar[-1,0] &&P_{-,-,-} \ar[-1,-2]
&& P_{+,-,-} \ar[-1,-2] & P_{-,+,-} \ar[-1,0]
}
\]
Thus we obtain the following lemma.

\begin{lem}
$\phi(C_{0,0,0})=\phi(P_{+,+,+}\cup P_{+,-,+})$.
\end{lem}

Suppose that $$\xi=\left(
\begin{array}{ccc}
1 &x_1 &0  \\
0 &1 &x_3  \\
x_2 &0 &1 
\end{array}
\right) \in P_{+,-,+}.$$
Then, from the nonsingularity condition, we have $x_1 x_2 x_3 =-1\pm 1$.
\begin{itemize}
 \item If $x_1 x_2 x_3 =0$, then $\xi \in P_{-,-,+}\cup P_{+,+,+} \cup P_{+,-,-}$.
 \item If $x_1 x_2 x_3 =-2$, then there exists $\psi\in GL_3(\bbZ)$ such that $(\psi,\tau_1)\cdot(E_3\ \xi)=(E_3\ \xi')$ where
$$\xi'=
\left(
\begin{array}{ccc}
1 &x_1 &0  \\
0 &1 &x_3  \\
-x_2 &-x_1 x_2 &1 
\end{array}
\right)\in C_{0,0,2}.
 $$
\end{itemize}
Thus we obtain the following lemma.

\begin{lem}\label{lem:C_{0,0,0}}
Put $A_1:=P_{+,+,+}$.
Then we have $\phi(C_{0,0,0})\subseteq \phi(A_1) \cup \phi(C_{0,0,2})$.
\end{lem}

For $C_{0,0,2}$, we prove the following lemma first. Put
$$C'_{0,0,2}:=\left\{\left(
\begin{array}{ccc}
1 &x_1 &x_2  \\
y_1 &1 &2  \\
y_2 &1 &1 
\end{array}
\right)\in C_{0,0,2}
\right\}.$$

\begin{lem}
$\phi(C'_{0,0,2})=\phi(C_{0,0,2})$.
\end{lem}

\begin{proof}
Since $(\delta_3,\sigma_3)$ gives a bijection between $\Xi(C_0,C_0,N_+)$ and $\Xi(C_0,C_0,N_-)$, we have $\phi(C_{0,0,2})=\phi(\Xi(C_0,C_0,N_+))$.
By using $(\nu_3,\mu_3)$, we see that $\phi(C'_{0,0,2})=\phi(\Xi(C_0,C_0,N_+))$.
\end{proof}

Suppose that $$\xi=\left(
\begin{array}{ccc}
1 &x_1 &0  \\
0 &1 &2  \\
y_2 &1 &1 
\end{array}
\right)\in C'_{0,0,2}.$$
From the nonsingularity condition, we have $2 x_1 y_2 =1\pm 1$.
\begin{itemize}
 \item If $x_1 y_2 =0$, then $\xi \in \Xi(P_+,P_+,N_+) \cup \Xi(P_-,P_-,N_+)$.
 \item If $x_1 y_2 =1$, then $x_1=\pm 1$ and
 $$\xi=\left(
\begin{array}{ccc}
1 &x_1 &0  \\
0 &1 &2  \\
x_1 &1 &1 
\end{array}
\right).$$
\end{itemize}
Similarly, if we assume $$\xi=\left(
\begin{array}{ccc}
1 &0 &x_2  \\
y_1 &1 &2  \\
0 &1 &1 
\end{array}
\right)\in C'_{0,0,2},$$
then we see that $\xi \in \Xi(P_+,P_+,N_+) \cup \Xi(P_-,P_-,N_+)$ or
$$\xi=\left(
\begin{array}{ccc}
1 &0 &2a  \\
a &1 &2  \\
0 &1 &1 
\end{array}
\right),\left(
\begin{array}{ccc}
1 &0 &b  \\
2b &1 &2  \\
0 &1 &1 
\end{array}
\right)$$
where $a$ and $b$ are $\pm 1$.
By using the action of $(\nu_1,\mu_1)$, we obtain the following.

\begin{lem}\label{lem:C_{0,0,2}}
Put $A_2:=C'_{0,0,2}\cap \Xi(P_+,P_+,N_+)$, $A_3:=C'_{0,0,2}\cap \Xi(P_-,P_-,N_+)$,
$$\chi_1:=\left(
\begin{array}{ccc}
1 &0 &2  \\
1 &1 &2  \\
0 &1 &1 
\end{array}
\right),\,\chi_2:=\left(
\begin{array}{ccc}
1 &0 &1  \\
2 &1 &2  \\
0 &1 &1 
\end{array}
\right),\,\chi_3:=\left(
\begin{array}{ccc}
1 &1 &0  \\
0 &1 &2  \\
1 &1 &1 
\end{array}
\right).$$
Then we have $\phi(C_{0,0,2})=\phi(A_2\cup A_3\cup \{\chi_1,\chi_2,\chi_3\})$.
\end{lem}

For $C_{0,2,2}$, we have the following.

\begin{lem}\label{lem:C_{0,2,2}}
Put
\begin{align*}
\chi_4 &:=\left(
\begin{array}{ccc}
1 &1 &1  \\
0 &1 &1  \\
2 &2 &1 
\end{array}
\right),
& \chi_5 &:=\left(
\begin{array}{ccc}
1 &2 &1  \\
0 &1 &1  \\
2 &2 &1 
\end{array}
\right),
& \chi_6 &:=\left(
\begin{array}{ccc}
1 &1 &1  \\
0 &1 &2  \\
2 &1 &1 
\end{array}
\right),
& \chi_7 &:=\left(
\begin{array}{ccc}
1 &2 &2  \\
0 &1 &1  \\
1 &2 &1 
\end{array}
\right), \\
\chi_8 &:=\left(
\begin{array}{ccc}
1 &4 &2  \\
0 &1 &1  \\
1 &2 &1 
\end{array}
\right),
& \chi_9 &:=\left(
\begin{array}{ccc}
1 &1 &2  \\
0 &1 &2  \\
1 &1 &1 
\end{array}
\right),
& \chi_{10} &:=\left(
\begin{array}{ccc}
1 &2 &2  \\
0 &1 &2  \\
1 &1 &1 
\end{array}
\right).
\end{align*}
Then we have $\phi(C_{0,2,2})=\phi(\{\chi_4,\ldots,\chi_{10}\})$.
\end{lem}

\begin{proof}
Take $\xi \in C_{0,2,2}$.
By using $(\delta_1,\sigma_1)$, $(\nu_2,\mu_2)$ and $(\nu_3,\mu_3)$, we can assume that $\xi$ is in one of the following forms:
$$ \xi=\left(
\begin{array}{ccc}
1 &x &1  \\
0 &1 &1  \\
2 &2 &1 
\end{array}
\right),\ \left(
\begin{array}{ccc}
1 &x &1  \\
0 &1 &2  \\
2 &1 &1 
\end{array}
\right),\ \left(
\begin{array}{ccc}
1 &x &2  \\
0 &1 &1 \\
1 &2 &1 
\end{array}
\right),\ \left(
\begin{array}{ccc}
1 &x &2  \\
0 &1 &2 \\
1 &1 &1 
\end{array}
\right).$$
If $\xi$ is in the first form, then we have $x=1,2$ by the nonsingularity condition.
Similarly, $x=1$ if $\xi$ is in the second form, $x=2,4$ in the third form, and $x=1,2$ in the forth form.
Thus we obtain the lemma.
\end{proof}

Similarly, we have the following lemma for $C_{2,2,2}$.

\begin{lem}\label{lem:C_{2,2,2}}
Put $$ \chi_{11}:=\left(
\begin{array}{ccc}
1 &2 &2  \\
1 &1 &2 \\
1 &1 &1 
\end{array}
\right).$$
Then we have $\phi(C_{2,2,2})=\{ \phi(\chi_{11}) \}$.
\end{lem}

\begin{proof}
Take $$ \xi=\left(
\begin{array}{ccc}
1 &x_1 &x_2  \\
y_1 &1 &x_3 \\
y_2 &y_3 &1 
\end{array}
\right)\in C_{2,2,2}. $$
By using $(\delta_1,\sigma_1)$, $(\nu_2,\mu_2)$ and $(\nu_3,\mu_3)$,  we can assume $x_1=2$, $y_1=1$ and $x_2,y_2>0$.
From the nonsingularity condition, we have $2 y_2 x_3 + x_2 y_3 = 5\pm 1$.
Then it is straightforward to obtain $$ \xi=\left(
\begin{array}{ccc}
1 &2 &2  \\
1 &1 &2 \\
1 &1 &1 
\end{array}
\right),\left(
\begin{array}{ccc}
1 &2 &2  \\
1 &1 &1 \\
1 &2 &1 
\end{array}
\right),\left(
\begin{array}{ccc}
1 &2 &1  \\
1 &1 &1 \\
2 &2 &1 
\end{array}
\right). $$
Moreover, if we put them $\xi_1$, $\xi_2$ and $\xi_3$, then we have $(\delta_3,\sigma_3)\cdot\xi_1=\xi_2$ and $(\delta_2,\sigma_2)\cdot\xi_2=\xi_3$.
\end{proof}

Taking $\tau_i$ ($i=1,2,3$) into account, we have the following diagram.
\[
\xymatrix {
\chi_1 \ar[1,0]_{\tau_3} \ar[1,1]^{\tau_2} \ar[0,1]^{\tau_1}
& \gamma_1 \ar[1,1]^{\tau_3} \ar[0,1]^{\tau_2}
& \gamma_2 \ar[0,1]^{\sigma_1}
& \chi_3 
& \chi_4 \ar[1,0]^{\tau_1}
& \chi_6 \ar[1,0]^{\tau_3} \\
\gamma_3 \ar[1,0]_{\sigma_3}
&\gamma_4 \ar[1,0]_{\tau_3} \ar[1,1]^{\sigma_1}
& \gamma_5 \ar[0,1]^{\sigma_3 \circ \sigma_2}
& \chi_2
& \gamma_6 \ar[1,0]^{\sigma_3 \circ \sigma_1}
& \gamma_7 \ar[1,0]^{\sigma_1} \\
\chi_7
& \chi_{11}
& \chi_9
&
& \chi_3
& \chi_8
}
\]
Here the arrow $\xi_1\xrightarrow{\rho}\xi_2$ means that there exist $\psi\in GL_3(\bbZ)$ and $\mu\in(\bbZ/2)^6$ such that $(\psi,\mu\circ\rho)\cdot(E_3\ \xi_1)=(E_3\ \xi_2)$, and $\gamma_i$ ($i=1,\dots,7$) denote the following characteristic squares:
\begin{align*}
\gamma_1 &:=\left(
\begin{array}{ccc}
1 &0 &2  \\
-1 &1 &0  \\
1 &1 &1 
\end{array}
\right),
& \gamma_2 &:=\left(
\begin{array}{ccc}
1 &0 &2  \\
1 &1 &0  \\
1 &1 &1 
\end{array}
\right),
& \gamma_3 &:=\left(
\begin{array}{ccc}
1 &2 &2  \\
1 &1 &2  \\
0 &1 &1 
\end{array}
\right),
& \gamma_4 &:=\left(
\begin{array}{ccc}
1 &0 &2  \\
1 &1 &2  \\
1 &1 &1 
\end{array}
\right), \\
\gamma_5 &:=\left(
\begin{array}{ccc}
1 &2 &2  \\
1 &1 &0  \\
0 &1 &1 
\end{array}
\right),
& \gamma_6 &:=\left(
\begin{array}{ccc}
1 &1 &1  \\
0 &1 &1  \\
2 &0 &1 
\end{array}
\right),
& \gamma_7 &:=\left(
\begin{array}{ccc}
1 &0 &1  \\
4 &1 &2  \\
2 &1 &1 
\end{array}
\right).
\end{align*}

Summarizing Lemma \ref{lem:first restriction}, Lemma \ref{lem:C_{0,0,0}}, Lemma \ref{lem:C_{0,0,2}}, Lemma \ref{lem:C_{0,2,2}}, Lemma \ref{lem:C_{2,2,2}}
and the above, we obtain the following.
Note that $\chi_3$ appears twice in the diagram and $\chi_5$, $\chi_{10}$ do not appear.

\begin{lem}\label{lem:summary up to w.e.homeo.}
$\calM^{\rm weh}_{I^3}= \phi(A_1\cup A_2\cup A_3\cup \{ \chi_1,\chi_5,\chi_6,\chi_{10} \})$.
\end{lem}


\begin{df}
We define $\xi_{s,t}\,(s,t\in\bbZ)$ by
$$\xi_{s,t}:=\left(
\begin{array}{ccc}
1 &s &t  \\
0 &1 &2  \\
0 &1 &1 
\end{array}
\right) \in A_2,$$
and $\xi^{s,t}\,(s,t\in\bbZ)$ by
$$\xi^{s,t}:=\left(
\begin{array}{ccc}
1 &0 &0  \\
s &1 &2  \\
t &1 &1 
\end{array}
\right) \in A_3.$$
Note that $A_2=\{\xi_{s,t}\}_{s,t\in\bbZ}$ and $A_3=\{\xi^{s,t}\}_{s,t\in\bbZ}$.
\end{df}

\section{Strong cohomological rigidity of $\calM_{I^3}$}


In this section, for $\xi\in\Xi_3$, we denote the generators of $H^*(\xi)$ by $X$, $Y$ and $Z$ instead of $X_1$, $X_2$ and $X_3$ (see Definition \ref{H^*(xi)}).
Additionally, we define $H^*(\xi;\bbZ/2):=H^*(\xi)/2$, $w_2(\xi):=\sum_{i=1}^6 u_i(\xi) \in H^2(\xi;\bbZ/2)$, $p_1(\xi):=-\sum_{i=1}^6 u_i(\xi)^2\in H^4(\xi)$ and identify $w_2(\xi)$, $p_1(\xi)$ with $w_2(M(\xi))$, $p_1(M(\xi))$ respectively through the canonical isomorphism between $H^*(\xi)$ and $H^*(M(\xi);\bbZ)$ (see Theorem \ref{thm:characteristic classes of a quasitoric manifold}).

\begin{df}
Let $\calM_{I^3}$ be the set of homeomorphism classes of quasitoric manifolds over $I^3$ and $\phi_1$ be the canonical surjection from $\calM^\mathrm{weh}_{I^3}$ to $\calM_{I^3}$.
We define subsets $\calM_1$, $\calM_2$ and $\calM_3$ of $\calM_{I^3}$ by 
$\calM_1:=\phi_1 \circ \phi (A_1)$, $\calM_2:=\phi_1 \circ \phi (A_2\setminus \{ \xi_{0,0} \})$ and $\calM_3:=\phi_1 \circ \phi (A_3)$.
In addition, we define $\calM^\mathrm{ceq}_{I^3}$ as the quotient $\calM_{I^3}/\!\sim$ where $M\sim M'$ if and only if $H^*(M;\bbZ)\cong H^*(M';\bbZ)$ as graded rings, and denote the quotient map by $\phi_2\co \calM_{I^3} \to \calM^\mathrm{ceq}_{I^3}$.
\end{df} 

\begin{df}
A class $\calC$ of topological spaces is called \textit{strongly cohomologically rigid} if for any graded ring isomorphism $\varphi$ between the cohomology rings of $X,Y\in\calC$ there exists a homeomorphism $f$ between them such that $\varphi=f^*$.
\end{df}

\begin{rem}\label{rem:strong cohomological rigidity of M_1}
By \cite[Proposition 6.2]{CMS10}, we see that $\calM_1$ corresponds with the class of 3-stage Bott manifolds.
Then we obtain the strong cohomological rigidity of $\calM_1$ by \cite[Theorem 3.1]{Choi} which shows the strong cohomological rigidity of 3-stage Bott manifolds.
\end{rem}

\begin{lem}\label{lem:M(chi_5),M(chi_6) in M_2}
$M(\chi_5),M(\chi_6),M(\chi_{10})\in\calM_2$.
Therefore $\calM_{I^3}=\calM_1\cup \calM_2\cup \calM_3\cup \{\chi_1\}$.
\end{lem}

\begin{proof}
Define graded ring automorphisms $\alpha_5,\alpha_6,\alpha_{10}$ of $\bbZ[X,Y,Z]$ so that
$$
\alpha_i\left(\begin{array}{c}X\\Y\\Z\end{array}\right)=
A_i\left(\begin{array}{c}X\\Y\\Z\end{array}\right)
$$
where $A_i$ ($i=5,6,10$) denote the following matrices:
$$A_5
:=\left(\begin{array}{ccc}
1 &0 &0  \\
0 &0 &1  \\
1 &1 &0 
\end{array}
\right),\,
A_6:=\left(
\begin{array}{ccc}
-1 &0 &0  \\
2 &1 &0  \\
0 &0 &1 
\end{array}
\right),\,
A_{10}:=\left(
\begin{array}{ccc}
1 &0 &0  \\
1 &1 &0  \\
0 &0 &1 
\end{array}
\right).
$$
Then these $\alpha_i$'s descend to isomorphisms $\alpha_5\co H^*(\xi_{-1,-2})\to H^*(\chi_5)$,
$\alpha_6\co H^*(\xi_{1,1})\to H^*(\chi_6)$, $\alpha_{10}\co H^*(\xi_{-2,-2})\to H^*(\chi_{10})$
and they preserve the second Stiefel-Whitney classes and the first Pontrjagin classes.
Thus we obtain the lemma by Theorem \ref{thm:classification theorem of 6-manifolds}.
\end{proof}

\begin{lem}
Let $\bbZ[Y,Z]$ be the polynomial ring generated by $Y$ and $Z$ of degree $2$, and $R$ be the quotient ring $\bbZ[Y,Z]/(Y(Y+2Z),Z(Y+Z))$.
Then $R$ has no non-zero element of degree $2$ of which the square is equal to $0$.
\end{lem}

\begin{proof}
Let $W=s Y+t Z$ be an element of which the square is $0$.
Then $0=W^2=(s Y+t Z)^2=(-2s^2+2s t-t^2)YZ=-\{s^2+(s-t)^2\}YZ$, so we have $s=s-t=0$, i.e. $W=0$.
\end{proof}

\begin{rem}\label{rem:nil-square element}
For any $\xi \in \Xi(\bbZ^2,\bbZ^2,C_2)$, since $H^*(\xi)/(X)$ is isomorphic to $R$,
the set $\{W\in H^2(\xi)\,|\,W^2=0\}$ is equal to $\bbZ X$ or $\{0\}$.
\end{rem}

This remark immediately yields the following lemma.

\begin{lem}\label{lem:M_1,M_3/M_2,M_4}
Let $M$ be a quasitoric manifold over $I^3$.
Then there exists non-zero $W \in H^2(M;\bbZ)$ such that $W^2=0$ if and only if $M\in \calM_1 \cup \calM_3$.
In particular, $\phi_2(\calM_1 \cup \calM_3) \cap \phi_2(\calM_2 \cup \{\chi_1\})=\emptyset$.
\end{lem}

\begin{lem}\label{lem:M_1/M_3}
$\phi_2(\calM_1)\cap\phi_2(\calM_3)=\emptyset$.
\end{lem}

\begin{proof}
Let $\xi_1 \in A_1,\,\xi_3 \in A_3$ and suppose that there exists an isomorphism $\alpha \co H^*(\xi_1)\to H^*(\xi_3)$.
Since $\alpha$ preserves the elements of which the squares are 0, $\alpha$ descends to an isomorphism
$\bar{\alpha}\co H^*(\xi_1)/(Z) \to H^*(\xi_3)/(X)$.
However, $H^*(\xi_1)/(Z)$ has non-zero degree 2 elements of which the squares are zero, but $H^*(\xi_3)/(X)\cong R$ does not.
This is a contradiction.
\end{proof}

\begin{lem}\label{lem:M_2/chi_1}
Let $\xi_0$ be a characteristic square on $I^{n-1}$ ($n\geq 3$), $\xi$ be a $(1,\xi_0)$-type characteristic square on $I^n$, and $\varphi\co \bbZ[X,Y,Z]\to \bbZ[X_1,\ldots,X_n]$ be a graded ring monomorphism which maps $X$, $Y$ and $Z$ to $\sum_{i=1}^n a_i X_i$, $\sum_{i=1}^n b_i X_i$ and $\sum_{i=1}^n c_i X_i$ respectively.
Moreover, we assume the following:
\begin{enumerate}
 \item[(a)] $\varphi$ maps $\calI_{\chi_1}$ into $\calI_{\xi}$;
 \item[(b)] 
for each prime $p$ the $\bmod p$ reductions of $(a_1,\ldots,a_n)$, $(b_1,\ldots,b_n)$ and $(c_1,\ldots,c_n)$ are linearly independent.
\end{enumerate}
Then we have $a_1=b_1=c_1=0$.
In particular, there exists no graded ring isomorphism from $H^*(\chi_1)$ to $H^*(\xi_{s,t})$ for any integers $s$ and $t$.
\end{lem}

\begin{proof}
Denote the first row of $\xi$ by $(1,s_2,s_3,\ldots,s_n)$.
Since $\xi$ is $(1,\xi_0)$-type and
\begin{align*}
& \varphi(X(X+2 Z)) = \left(\sum_{i=1}^n a_i X_i\right)\left\{ \sum_{i=1}^n (a_i+2 c_i) X_i \right\} \\
& = X_1 \left\{ a_1(a_1+2 c_1) X_1 + \sum_{i=2}^n \{a_1(a_i+2 c_i)+a_i(a_1+2 c_1)\} X_i \right\}+(\text{a polynomial in }X_2,\ldots,X_n) \\
& = X_1 \sum_{i=2}^n \{ a_1(a_i+2 c_i)+(a_i-s_i a_1)(a_1+2 c_1) \} X_i +(\text{a polynomial in }X_2,\ldots,X_n)=0
\end{align*}
in $H^*(\xi)$, we obtain $a_1(a_i+2 c_i)=(s_i a_1 -a_i)(a_1+2 c_1)$ for $i=2,\ldots,n$.
Note that, by the assumption of linear independence, $a_1+2c_1=0$ if $a_1=0$ and vice versa.
If $a_1$ and $a_1+2 c_1$ are non-zero, denoting by $k$ the greatest common divisor of $a_1$ and $a_1+2 c_1$, we see that $a_1/k$ and $(a_1+2 c_1)/k$ divide $s_i a_1 -a_i$ and $a_i +2 c_i$ ($i=2,\ldots,n$) respectively.
By the assumption (b), we obtain $a_1/k=\pm 1$ and $(a_1+2 c_1)/k=\pm 1$,
namely, $a_1+c_1=0$ or $c_1=0$.
This holds also in the case $a_1=a_1+2 c_1=0$.

Similarly, we have $b_1(a_i+b_i+2 c_i)=(s_i b_1-b_i)(a_1+b_1+2 c_1)$ and $c_1(b_i+ c_i)=(s_i c_1-c_i)(b_1+ c_1)$ for $i=2,\ldots,n$ from $\varphi(Y(X+Y+2 Z))=0$ and $\varphi(Z(Y+Z))=0$ respectively,
and then obtain $a_1+2 c_1=0$ or $a_1+2 b_1+2 c_1=0$, and, $b_1=0$ or $b_1+2 c_1=0$ in the same way as above.
We solve these equations to see $a_1=b_1=c_1=0$.
\end{proof}

\begin{rem}\label{rem:remark on proof over I^3}
By Remark \ref{rem:strong cohomological rigidity of M_1}, Lemma \ref{lem:M(chi_5),M(chi_6) in M_2}, Lemma \ref{lem:M_1,M_3/M_2,M_4}, Lemma \ref{lem:M_1/M_3} and Lemma \ref{lem:M_2/chi_1},
to show the strong cohomological rigidity of $\calM_{I^3}$,
we only have to show that of $\calM_2$, $\calM_3$, and $\{ M(\chi_1) \}$ respectively.
\end{rem}


\begin{lem}\label{filtration lemma for M_2}
Let $\xi_0$ be a characteristic square on $I^{n-1}$ ($n\geq 2$), $\xi$ be a $(1,\xi_0)$-type characteristic square on $I^n$, and $\varphi\co \bbZ[X_1,X_2]\to \bbZ[X_1,\ldots,X_n]$ be a graded ring monomorphism which maps $X_1$ and $X_2$ to $\sum_{i=1}^n a_i X_i$ and $\sum_{i=1}^n b_i X_i$ respectively.
Moreover, we assume the following:
\begin{enumerate}
 \item[(a)] $\varphi$ maps $\calI_{\kappa^2}$ into $\calI_{\xi}$;
 \item[(b)] 
for any prime $p$ the $\bmod p$ reductions of $(a_1,\ldots,a_n)$ and $(b_1,\ldots,b_n)$ 
are linearly independent.
\end{enumerate}
Then we have $a_1=b_1=0$.
\end{lem}

\begin{proof}
Denote the first row of $\xi$ by $(1,s_2,s_3,\ldots,s_n)$.
Since $\xi$ is $(1,\xi_0)$-type and
\begin{align*}
& \varphi(X_1(X_1+2 X_2)) = \left(\sum_{i=1}^n a_i X_i\right)\left\{ \sum_{i=1}^n (a_i+2 b_i) X_i \right\} \\
& = X_1 \left\{ a_1(a_1+2 b_1) X_1 + \sum_{i=2}^n \{a_1(a_i+2 b_i)+a_i(a_1+2 b_1)\} X_i \right\}+(\text{a polynomial in }X_2,\ldots,X_n) \\
& = X_1 \sum_{i=2}^n \{ a_1(a_i+2 b_i)+(a_i-s_i a_1)(a_1+2 b_1) \} X_i  +(\text{a polynomial in }X_2,\ldots,X_n)=0
\end{align*}
in $H^*(\xi)$, we obtain $a_1(a_i+2 b_i)=(s_i a_1 -a_i)(a_1+2 b_1)$ for $i=2,\ldots,n$.
In the same way as the proof of Lemma \ref{lem:M_2/chi_1}, we obtain $a_1+b_1=0$ or $b_1=0$, which implies that the coefficient of $X_1$ in $\varphi(X_1+X_2)$ or $\varphi(X_2)$ is zero.
Then we easily see that $\varphi(X_2(X_1+X_2))\neq 0$ in $H^*(\xi)$ unless both $b_1$ and $a_1+b_1$ are zero.
\end{proof}

\begin{lem}\label{lem:strong cohomological rigidity of M_2}
$\calM_2$ is strongly cohomologically rigid.
\end{lem}

\begin{proof}
Let $\varphi\co H^*(\xi_{s,t})\to H^*(\xi_{x,y})$ be a graded ring isomorphism.
By Lemma \ref{filtration lemma for M_2},
$$\varphi\left(\begin{array}{c}X\\Y\\Z\end{array}\right)=\left(
\begin{array}{c | c c}
a &b &c \\ \hline
0 & \multicolumn{2}{c}{\raisebox{-11pt}[0pt][0pt]{\LARGE $\theta$}} \\
0 &
\end{array}
\right)\left(\begin{array}{c}X\\Y\\Z\end{array}\right)$$
where $a=\pm 1$ and $\theta$ is an automorphism of $H^*(\kappa^2)$.
By Lemma \ref{realization of Aut(kappa^2)}, $\theta$ can be realized as a weakly equivariant self-homeomorphism of $M(\kappa^2)$, and therefore we can construct a weakly equivariant homeomorphism $f$ from $M(\xi_{x,y})$ to some $M(\xi_{x',y'})$ such that
$$f^*\left(\begin{array}{c}X\\Y\\Z\end{array}\right)=\left(
\begin{array}{c | c c}
a &0 &0 \\ \hline
0 & \multicolumn{2}{c}{\raisebox{-11pt}[0pt][0pt]{\LARGE $\theta$}} \\
0 &
\end{array}
\right)\left(\begin{array}{c}X\\Y\\Z\end{array}\right)$$
in the similar way to the proof of Corollary \ref{1-I}.
Thus we see that we can assume $a=1$ and $\theta=E_2$.
Since $\varphi$ maps $\calI_{\xi_{s,t}}$ into $\calI_{\xi_{x,y}}$,
\begin{align*}
& \varphi(X(X+s Y+t Z)) = (X+ b Y+c Z)\{X+(s+b)Y+(t+c)Z\} \\
& = X \{ (s-x+2 b)Y+(t-y+2 c)Z \}+\{ -2 b(s+b) -c(t+c) +b(t+c) +c(s+b) \} YZ \\
& =0
\end{align*}
in $H^*(\xi_{x,y})$.
Thus we obtain
$$b=\frac{x-s}{2},\,c=\frac{y-t}{2},\,(s-t)^2+s^2=(x-y)^2+x^2.$$
In particular, $s\equiv x$ and $t\equiv y$ modulo 2.
Then we have
$$ \varphi(w_2(\xi_{s,t}))=\varphi((s+1) Y+t Z))=(s+1)Y+t Z=w_2(\xi_{x,y}) $$
in $H^*(\xi_{x,y};\bbZ/2)$.
Similarly, since $\varphi(2 X+sY+t Z)-(2 X+x Y+y Z)=0$, we have
\begin{align*}
& p_1(\xi_{x,y}) - \varphi(p_1(\xi_{s,t})) = \varphi(X)^2+\varphi(X+sY+t Z)^2-X^2-(X+x Y+y Z)^2 \\
& = \varphi(2 X+sY+t Z)^2-(2 X+x Y+y Z)^2 \\
& = \{ \varphi(2 X+sY+t Z)+(2 X+x Y+y Z) \} \{ \varphi(2 X+sY+t Z)-(2 X+x Y+y Z) \} \\
& = 0
\end{align*}
in $H^*(\xi_{x,y})$.
Thus we obtain the lemma by Theorem \ref{thm:classification theorem of 6-manifolds}.
\end{proof}

\begin{lem}\label{lem:strong cohomological rigidity of M_3}
Any graded ring isomorphism between the cohomology rings of two members of $\phi(A_3)$ is induced by a weakly equivariant homeomorphism.
In particular, $\calM_3$ is strongly cohomologically rigid.
\end{lem}

\begin{proof}
Note that $(\psi_3\psi_2,\sigma_3\circ\sigma_2)\cdot\xi^{s,t}$ is a $(\kappa^2,1)$-type characteristic square.
Let $\xi$ and $\xi'$ be two $(\kappa^2,1)$-type characteristic squares and  $\varphi\co H^*(\xi)\to H^*(\xi')$ be a graded ring isomorphism.
Since $\varphi$ preserves the elements of degree 2 of which the squares are zero, we have
$$\varphi\left(\begin{array}{c}X\\Y\\Z\end{array}\right)=\left(
\begin{array}{c c | c}
\multicolumn{2}{c|}{\raisebox{-11pt}[0pt][0pt]{\LARGE $\theta$}} & a\\
&& b \\ \hline
0 & 0 & c
\end{array}
\right)\left(\begin{array}{c}X\\Y\\Z\end{array}\right).$$
As in the proof of the previous lemma, we can assume $c=1$ and $\theta=E_2$.
Then we have $\xi=\xi'$ and $\varphi=(\id_{M(\xi)})^*$ by Lemma \ref{lemma for 1-II}.
\end{proof}



\begin{lem}\label{lem:automorphisms of chi_1}
Let $\varphi$ be a graded ring automorphism of $H^*(\chi_1)$.
Then $\varphi= \pm \id $.
\end{lem}

\begin{proof}
Take $A\in GL_3(\bbZ)$ so that
$$\varphi\left(\begin{array}{c}X\\Y\\Z\end{array}\right)
=A\left(\begin{array}{c}X\\Y\\Z\end{array}\right)$$
and denote the $i$-th row of $A$ by $(a_i,b_i,c_i)$ ($i=1,2,3$).
Then $A$ satisfies the following equations.
\begin{align} \label{eq:chi_1,1}
(a_1-b_1)(b_1+2 b_3) &= -b_1(a_1+2 a_3) \\
(c_1-2b_1)(b_1+2 b_3) &= (c_1-b_1)(c_1+2 c_3) \\
(c_1-2a_1)(a_1+2a_3) &= -a_1(c_1+2c_3) \\
(a_2-b_2)(b_1+b_2+2 b_3) &= -b_2(a_1+a_2+2 a_3) \\
(c_2-2b_2)(b_1+b_2+2 b_3) &= (c_2-b_2)(c_1+c_2+2 c_3) \\
(c_2-2a_2)(a_1+a_2+2a_3) &= -a_2(c_1+c_2+2c_3) \\
(a_3-b_3)(b_2+b_3) &= -b_3(a_2+a_3) \\
(c_3-2b_3)(b_2+b_3) &= (c_3-b_3)(c_2+c_3) \\ \label{eq:chi_1,9}
(c_3-2a_3)(a_2+a_3) &= -a_3(c_2+c_3)
\end{align}
By solving these equations modulo $2$, we obtain
$$A \equiv \left(
\begin{array}{ccc}
1 &0 &0  \\
0 &1 &0  \\
0 &b_3 &1 
\end{array}
\right) \bmod 2.
$$
Since $a_1$ is odd and $b_1$ is even, we have 
$$ \frac{b_1}{2}+b_3\equiv -\frac{b_1}{2} \bmod 2$$
from the equation (\ref{eq:chi_1,1}), which implies $b_3\equiv 0\bmod 2$.

Moreover, we obtain $c_2-b_2,\,c_3-b_3,\,c_3-2b_3=\pm 1$ as follows.
Note that $c_2-b_2$, $c_3-b_3$, and $c_3-2b_3$ are odd.
Let $p$ be an odd prime, and consider the equations $(\ref{eq:chi_1,1}),\ldots,(\ref{eq:chi_1,9})$ and $\det A = \pm 1$ modulo $p$.
Then, by a direct calculation, one can show that there exists no solution with 
$c_2-b_2\equiv 0$, $c_3-b_3\equiv 0$, or $c_3-2b_3\equiv 0$ modulo $p$.
This implies that no prime divides them, i.e. they are equal to $\pm 1$ respectively.
Then we can solve $(\ref{eq:chi_1,1}),\ldots,(\ref{eq:chi_1,9})$ straightforwardly and obtain the lemma.
\end{proof}


The following theorem is immediate from Remark \ref{rem:remark on proof over I^3}, Lemma \ref{lem:strong cohomological rigidity of M_2},
Lemma \ref{lem:strong cohomological rigidity of M_3} and Lemma \ref{lem:automorphisms of chi_1}, which is a paraphrase of Theorem \ref{main thm 2}.

\begin{thm}
$\calM^{\rm homeo}_{I^3}$ is strongly cohomologically rigid.
\end{thm}


\end{document}